\RequirePackage{latexsym, amsmath, amstext, amssymb, amsfonts, amscd, bm, array, amsbsy, mathrsfs, amscd}
\documentclass[11pt, cmp, numbook, final, a4paper]{svjour}
\usepackage[11pt]{extsizes}
\pdfoutput=1
\usepackage{indentfirst}
\usepackage{graphics} 
\usepackage{epsfig}
\usepackage{transparent, colortbl}
\usepackage{booktabs}
\usepackage{multirow}
\usepackage{anysize}
\usepackage{latexsym}
\usepackage{pdfsync}
\usepackage[all]{xy}
\usepackage[usenames,dvipsnames]{xcolor}
\usepackage{enumerate}
\usepackage{float}
\restylefloat{table}
\usepackage{upgreek}
\usepackage{array}
\def\Khr{K\"{a}hlerian~}

\def\rGamma{\mathrm{\Gamma}}
\def\rH{\mathrm{H}}
\def\Im {\mathrm{Im}}

\def\NS{\mathrm{NS}}

\usepackage[colorlinks=true,linkcolor=black]{hyperref}
\colorlet{linkequation}{blue}

\newcommand*{\SavedEqref}{}
\let\SavedEqref\eqref
\renewcommand*{\eqref}[1]{%
  \begingroup
    \hypersetup{
      linkcolor=blue,
      linkbordercolor=blue,
    }%
    \SavedEqref{#1}%
  \endgroup
}

\DeclareSymbolFont{extraup}{U}{zavm}{m}{n}
\DeclareMathSymbol{\varheart}{\mathalpha}{extraup}{86}
\DeclareMathSymbol{\vardiamond}{\mathalpha}{extraup}{87}

\makeatletter
\def\CT@@do@color{%
  \global\let\CT@do@color\relax
        \@tempdima\wd\z@
        \advance\@tempdima\@tempdimb
        \advance\@tempdima\@tempdimc
        \kern-\@tempdimb
\transparent{0.6}%
        \leaders\vrule
                \hskip\@tempdima\@plus  1fill
        \kern-\@tempdimc
        \hskip-\wd\z@ \@plus -1fill }
\makeatother

\makeatletter
\newcommand{\thickhline}{%
    \noalign {\ifnum 0=`}\fi \hrule height 1pt
    \futurelet \reserved@a \@xhline
}
\newcolumntype{"}{@{\hskip\tabcolsep\vrule width 1pt\hskip\tabcolsep}}
\makeatother

\newtheorem{Theorem}{Theorem}[section]
\newtheorem{Lemma}[Theorem]{Lemma}
\newtheorem{Proposition}[Theorem]{Proposition}
\newtheorem{Corollary}[Theorem]{Corollary}

\newtheorem{Definition}[Theorem]{Definition}

\newcommand{\Hom}{{\rm Hom}}

\newcommand{\Spec}{\mathrm{Spec}}

\newcommand{\im}{\mathrm{im}}

\newcommand{\Coh}{\mathrm{Coh}}

\newcommand{\Proj}{\mathrm{Proj}}

\newcommand{\Ext}{\mathrm{Ext}}

\newcommand{\Jac}{\mathrm{Jac}}

\newcommand{\bp}{\begin{Proposition}}
\newcommand{\ep}{\end{Proposition}}
\newcommand{\bl}{\begin{Lemma}}
\newcommand{\el}{\end{Lemma}}
\newcommand{\bt}{\begin{Theorem}}
\newcommand{\et}{\end{Theorem}}
\newcommand{\bd}{\begin{Definition}}
\newcommand{\ed}{\end{Definition}}
\newcommand{\End}{\mathrm{End}}

\newcommand{\Mod}{\mathrm{Mod}}

\newcommand{\Mat}{\mathrm{Mat}}
\newcommand{\ev}{\mathrm{ev}}
\newcommand{\eqdef}{\stackrel{{\rm def.}}{=}}

\newcommand{\cinf}{{C^\infty(X)}}

\let\SSS\S

\DeclareFontFamily{U}{rsf}{}
\DeclareFontShape{U}{rsf}{m}{n}{<5> <6> rsfs5 <7> <8> <9> rsfs7 <10-> rsfs10}{}
\DeclareMathAlphabet\Scr{U}{rsf}{m}{n}

\def\cW{\mathcal{W}}

\def\cE{\check{E}}

\def\bE{\mathbf{E}}

\def\Z{\mathbb{Z}}
\def\C{\mathbb{C}}
\def\R{\mathbb{R}}

\def\H{\mathbb{H}}

\def\S{\mathbb{S}}
\def\rk{{\rm rk}}

\def\deg{{\rm deg}}

\def\dd{\mathrm{d}}

\def\fd{\mathfrak{d}}
\def\md{\boldsymbol{\fd}}

\def\cK{{\cal K}}

\def\codim{\mathrm{codim}}

\def\F{\mathrm{F}}
\def\i{\mathbf{i}}
\def\cC{\mathcal{C}}

\def\MF{\mathrm{MF}}
\def\HMF{\mathrm{HMF}}

\def\bK{\mathbf{K}}
\def\Prin{\mathrm{Prin}}
\def\rU{\mathrm{U}}

\newcommand{\be}{\begin{equation*}}
\newcommand{\ee}{\end{equation*}}
\newcommand{\ben}{\begin{equation}}
\newcommand{\een}{\end{equation}}
\newcommand{\beqa}{\begin{eqnarray*}}
\newcommand{\eeqa}{\end{eqnarray*}}
\newcommand{\beqan}{\begin{eqnarray}}
\newcommand{\eeqan}{\end{eqnarray}}

\newcommand \pd {{\partial}}
\newcommand \bpd {{\overline{\partial}}}

\newcommand{\nn}{\nonumber}

\newcommand{\id}{\mathrm{id}}

\def\bvarphi{\boldsymbol{\varphi}}

\def\cL{\mathrm{Cl}}

\def\cK{\mathrm{\cal K}}
\def\cQ{\mathrm{\cal Q}}

\def\odd{\mathrm{odd}}

\def\O{\mathrm{O}}

\def\cA{\mathcal{A}}
\def\cE{\mathcal{E}}

\def\cP{\mathcal{P}}
\def\bP{\mathbf{P}}

\def\cT{\mathcal{T}}
\def\cF{\mathcal{F}}

\def\G_2{\mathrm{G_2}}
\def\cO{\mathcal{O}}
\def\cL{\mathcal{L}}

\def\P{\mathbb{P}}

\def\VB{\mathrm{VB}}

\def\sc{\mathrm{sc}}

\def\HF{\mathrm{HF}}

\def\PV{\mathrm{PV}}
\def\HPV{\mathrm{HPV}}

\def\ioda{\boldsymbol{\iota}}
\def\bbpd{\boldsymbol{\bpd}}

\def\0{{\hat{0}}}
\def\1{{\hat{1}}}

\def\cJ{\mathcal{J}}
\def\Jac{\mathrm{Jac}}

\def\PF{\mathrm{PF}}
\def\HPF{\mathrm{HPF}}
\def\O{\mathrm{O}}
\def\DF{\mathrm{DF}}
\def\HDF{\mathrm{HDF}}

\def\mGamma{\boldsymbol{\rGamma}}
\def\J{\mathrm{J}}

\def\mod{\mathrm{mod}}

\def\proj{\mathrm{proj}}

\def\free{\mathrm{free}}

\def\triv{\mathrm{triv}}
\def\alg{\mathrm{alg}}
\def\an{\mathrm{an}}
\def\Pic{\mathrm{Pic}}

\def\rR{\mathrm{R}}

\def\rS{\mathrm{S}}

\begin{document}

\title{On B-type open-closed Landau-Ginzburg theories defined on Calabi-Yau Stein manifolds}

\author{Elena Mirela Babalic \and  Dmitry Doryn \and Calin Iuliu Lazaroiu \and Mehdi
  Tavakol}

\institute{Center for Geometry and Physics, Institute for Basic
  Science, Pohang  37673, Republic of Korea\\
\email{mirela@ibs.re.kr, dmitry@ibs.re.kr, calin@ibs.re.kr, mehdi@ibs.re.kr}}

\date{}

\maketitle

\abstract{We consider the bulk algebra and topological D-brane
category arising from the differential model of the open-closed B-type
topological Landau-Ginzburg theory defined by a pair $(X,W)$, where
$X$ is a non-compact Calabi-Yau manifold and $W$ is a complex-valued
holomorphic function. When $X$ is a Stein manifold (but not restricted
to be a domain of holomorphy) we extract equivalent descriptions of
the bulk algebra and of the category of topological D-branes which are
constructed using only the analytic space associated to $X$. In
particular, we show that the D-brane category is described by
projective factorizations defined over the ring of holomorphic
functions of $X$. We also discuss simplifications of the analytic
models which arise when $X$ is holomorphically parallelizable and
illustrate these in a few classes of examples.}

\tableofcontents 

\section{Introduction}
\label{sec:intro}

Quantum oriented open-closed B-type Landau-Ginzburg theories are
non-anomalous quantum field theories conjecturally associated to pairs
$(X,W)$, where $X$ is a non-compact Calabi-Yau manifold  (non-compact 
\Khr manifold with trivial canonical line bundle) and
$W:X\rightarrow \C$ is a non-constant holomorphic function.  Such
theories should be obtained by quantization of the classical models
constructed in \cite{LG1,LG2}. A
successful quantization must in particular associate to each pair
$(X,W)$ an oriented two-dimensional open-closed TFT (topological field
theory) in the sense axiomatized in \cite{tft,MS,LP}, which in turn is
equivalent with an algebraic structure known as a ``TFT datum''. When
the critical set of $W$ is compact, non-rigorous path integral
arguments imply \cite{LG2} that the TFT datum of such theories can be
recovered from the cohomology of any member of a family of
differential models built using a certain dg (differential graded)
algebra $(\PV(X),\updelta_W)$ and a certain dg-category $\DF(X,W)$ of
bundle-valued differential forms defined on $X$ (see \cite{nlg1,tserre} for a
rigorous construction of such models).

As already pointed out in \cite{LG2}, the differential models of the
TFT datum associated to $(X,W)$ are quasi-isomorphic to each other and
their cohomology admits an equivalent analytic description which is
controlled by certain spectral sequences whose limits depend markedly
on the geometry of $X$ and on the nature of the critical locus of
$W$. For example, when $X=\C^n$ and $W$ has finite critical set, it
was argued in \cite{LG2} that such spectral sequences provide an
equivalent analytic model of the TFT datum which is realized in terms
of residues \cite{Vafa,KL1}, a model which was studied rigorously in
\cite{D,DM,PV}. When $X$ is a domain of holomorphy in $\C^n$, some
aspects of the differential and analytic models were studied in
\cite{LLS,S}.

In this paper, we consider the problem of constructing equivalent
analytic models for the TFT datum in the more general setting when $X$
is an arbitrary Calabi-Yau Stein manifold\footnote{Notice that any
Stein manifold is K\"ahlerian, since it admits a holomorphic embedding
in some complex affine space $\C^N$, whose K\"ahler form restricts to
$X$.} (not restricted to be a domain of holomorphy). We first show
that, without any restrictions on $X$, the cohomological algebra
$\HPV(X,W)\eqdef \rH(\PV(X),\updelta_W)$ admits an isomorphic analytic
model constructed as the hypercohomology of the Koszul complex of the
holomorphic 1-form $-\i \pd W$. When $X$ is Stein and $W$ has isolated
critical points, the latter reduces to the global Jacobi algebra of
$W$, whose description simplifies further when $X$ is holomorphically
parallelizable. When $X$ is Stein, we also show that the cohomological
category $\HDF(X,W)\eqdef \rH(\DF(X,W))$ is equivalent with a simpler
category $\HF(X,W)\eqdef\rH(\F(X,W))$, where $\F(X,W)$ is another
dg-category introduced in \cite{nlg1}, which is defined using only
analytic data. Through the Serre-Swan correspondence for Stein
manifolds \cite{ForsterStein,Morye}, the category $\HF(X,W)$ is itself
equivalent with a category of {\em analytic projective factorizations} $\HPF(X,W)$. A
projective factorization is a pair $(P,D)$, where $P$ is a
finitely-generated $\Z_2$-graded module over the algebra $\O(X)$ of
complex-valued holomorphic functions defined on $X$ and $D$ is an odd
endomorphism of $P$ which squares to $W\id_P$. In the Stein case, we
also obtain analytic models of the cohomological disk algebra
$\HPV(X,a)\eqdef\rH(\PV(X,a),\Delta_a)$ of a holomorphic factorization
$a$ of $W$, which is defined \cite{nlg1} as the cohomology of a
certain dg-algebra $(\PV(X,a),\Delta_a)$. Finally, we illustrate these
analytic models in a few classes of examples.

The paper is organized as follows. Section \ref{sec:LGpairs} describes
the basic analytic objects associated to a pair $(X,W)$. Section
\ref{sec:diffmodel} recalls the definition \cite{nlg1} of the
dg-algebra $(\PV(X),\updelta_W)$ and of the dg-categories $\DF(X,W)$
and $\F(X,W)$, as well as of the differential graded disk algebra
$(\PV(X,a),\Delta_a)$. Sections \ref{sec:bulk}, \ref{sec:boundary} and
\ref{sec:disk} discuss the equivalent analytic models of $\HPV(X,W)$,
$\HDF(X,W)$ and $\HPV(X,a)$. Section \ref{sec:examples} illustrates
these analytic models in a few classes of examples. Section
\ref{sec:conclusions} concludes and discusses some further directions,
placing the current work in the wider physics and mathematics context
of Mirror Symmetry. Appendix \ref{app:Stein} summarizes some relevant
properties of Stein manifolds. The reader who is unfamiliar with Stein
geometry (and in particular with Cartan's theorem B, which plays a
crucial role in the proof of some results herein) is encouraged to
refer to the appendix.

\subsection{Notations and conventions}
\label{subsec:notation}

We use the notations and conventions of \cite[Subsection 1.1]{nlg1}.
Given a commutative ring $R$, let $\Mod_R$ denote the category of
$R$-modules and $\mod_R$ denote the full sub-category of
finitely-generated $R$-modules. Let $\Proj_R$ denote the full
subcategory of $\Mod_R$ consisting of projective $R$-modules and
$\proj_R$ denote the full subcategory of $\mod_R$ consisting of
finitely-generated projective $R$-modules. All manifolds considered
are smooth, paracompact, {\em connected} and of non-zero dimension and
all vector bundles considered are smooth.

\section{Landau-Ginzburg pairs}
\label{sec:LGpairs}

\begin{Definition}
A {\em Landau-Ginzburg (LG) pair} of dimension $d$ is a pair $(X,W)$,
where:
\begin{enumerate}[A.]
\itemsep 0.0em
\item $X$ is a non-compact K\"{a}hlerian manifold of complex dimension
  $d$ which is {\em Calabi-Yau} in the sense that the canonical line
  bundle $K_X\eqdef \wedge^d T^\ast X$ is holomorphically trivial.
\item $W:X\rightarrow \C$ is a {\em non-constant} complex-valued
  holomorphic function defined on $X$.
\end{enumerate}
The {\em signature} $\mu(X,W)$ is the mod 2 reduction of $d$:
\be
\mu(X,W)\eqdef {\hat d}\in \Z_2~~.
\ee
\end{Definition}

\noindent Let $(X,W)$ be a Landau-Ginzburg pair. Let $\cO_X$ denote
the sheaf of locally-defined complex-valued holomorphic functions and
$\O(X)=\rGamma(X,\cO_X)=\rH^0(\cO_X)$ denote the ring
of globally-defined holomorphic functions from $X$ to $\C$. Let:
\ben
\label{iodaW}
\ioda_W\eqdef -\i (\pd W)\lrcorner: TX\rightarrow \cO_X
\een
 denote the
morphism of sheaves of $\cO_X$-modules given by left contraction with
$-\i \pd W$, where we identify the holomorphic tangent bundle $TX$
with its locally-free sheaf of holomorphic sections.

\

\begin{Definition} 
The {\em critical set of $W$} is defined through: 
\be
Z_W\eqdef \{p\in X|(\pd W)(p)=0\}~~.
\ee
The {\em critical sheaf} of $W$ is the ideal sheaf:
\be
\cJ_W\eqdef \im(\ioda_W:TX\rightarrow \cO_X)\subset \cO_X~~.
\ee
The {\em critical ideal} of $(X,W)$ is the following ideal of the
commutative ring $\O(X)$:
\be
\J(X,W)\eqdef \cJ_W(X)=\ioda_W(\rGamma(X,TX))\subset \O(X)~~.
\ee
\end{Definition}

\

\noindent Notice that $\i\partial W\in \rGamma(X,T^\ast X)$ is a
holomorphic section of the holomorphic cotangent bundle $T^\ast X$ and
that $Z_W$ is the vanishing locus of this section. For any open subset
$U\subset X$ supporting local complex coordinates
$z^1,\ldots,z^d$, the ideal $\cJ_W(U)\subset \cO_X(U)$ is generated by
the partial derivatives $\frac{\partial W}{\partial z^1},\ldots,
\frac{\partial W}{\partial z^d}\in \cO_X(U)$.

\

\begin{Definition}
\label{def41} The {\em Jacobi sheaf} of $W$ is the sheaf of 
commutative $\cO_X$-algebras defined through:
\be
Jac_W\eqdef \cO_X/\cJ_W~~.
\ee
The \emph{Jacobi algebra} $\Jac(X,W)$ of the LG pair $(X, W)$ is the
commutative $\O(X)$-algebra of globally-defined sections of the Jacobi
sheaf:
\be
\Jac(X,W)\eqdef Jac_W(X)= \rH^0(Jac_W)~~.
\ee
\end{Definition}

\noindent Let $\cO_X\rightarrow Jac_W$ denote the projection map. 
The Jacobi sheaf is supported on the critical set $Z_W$ and the 
restriction $\cO_{Z_W}\eqdef Jac_W|_{Z_W}$ makes $Z_W$ into an analytic
subspace $(Z_W,\cO_{Z_W})$ of the analytic space $(X,\cO_X)$, which we
call the {\em Jacobi space} of $W$.

\

\begin{Definition}
The {\em sheaf Koszul complex} of $W$ is the following complex of locally-free sheaves
of $\cO_X$-modules:
\ben
\label{Koszul}
(\cK_W):~0 \rightarrow \wedge^d TX \stackrel{\ioda_W}{\rightarrow} \wedge^{d-1} TX 
\stackrel{\ioda_W}{\rightarrow} \ldots \stackrel{\ioda_W}{\rightarrow} \cO_X \rightarrow 0~~,
\een
where $\cO_X$ sits in degree zero and we identify the exterior power $\wedge^k T X$ with its
locally-free sheaf of holomorphic sections.
\end{Definition}

\

\noindent With our convention, $\cK_W$ is concentrated in non-positive degrees.

\

\begin{Proposition}
\label{prop:Koszul}
Assume that the critical set of $W$ is finite, i.e. $\dim_\C Z_W=0$. Then the sequence: 
\ben
\label{KoszulRes}
0~ \rightarrow \wedge^d TX \stackrel{\ioda_W}{\rightarrow} \wedge^{d-1} TX \stackrel{\ioda_W}{\rightarrow} \ldots \stackrel{\ioda_W}{\rightarrow} 
TX \stackrel{\ioda_W}{\rightarrow} \cO_X \rightarrow Jac_W\rightarrow 0
\een
is exact and thus provides a resolution of $Jac_W$ through locally
free sheaves of $\cO_X$-modules. In particular, the sheaf Koszul
complex \eqref{Koszul} is exact except at the last term.
\end{Proposition}

\begin{proof}
Since $\dim_\C Z_W=0$, we have $\codim_\C Z=d=\dim_\C T^\ast X$, which
means that $-\i\partial W$ is a regular section of $T^\ast X$. Hence
the sheaf sequence:
\be
0~ \rightarrow \wedge^d TX \stackrel{\ioda_W}{\rightarrow} \wedge^{d-1} TX \stackrel{\ioda_W}{\rightarrow} \ldots \stackrel{\ioda_W}{\rightarrow} TX \stackrel{\ioda_W}{\rightarrow} \cJ_W \rightarrow 0
\ee
is exact in the Abelian category $\Coh(X)$ of coherent sheaves of $\cO_X$-modules,
 being a resolution of the
ideal sheaf $\cJ_W$ of $(Z_W,\cO_{Z_W})$ (see \cite[page 5]{AN}). Since $\cJ_W=\ker(\cO_X \rightarrow
Jac_W)$,~this implies the conclusion. ~\qed
\end{proof}

\begin{remark}
\label{rem:SteinFinite}
Assume that $X$ is Stein. Then the critical set $Z_W$ is compact iff
it is a finite set. Indeed, $Z_W$ is a subvariety of $X$ and a Stein
manifold does not admit compact subvarieties of positive dimension (see 
 \cite[Chap V.4, Theorem 3]{GR}).
\end{remark}

\section{Some structures of the differential model}
\label{sec:diffmodel}

\noindent Let $(X,W)$ be a Landau-Ginzburg pair of dimension
$d$. Using path integral arguments, it was argued in \cite{LG2} that,
when the critical set $Z_W$ is compact, the TFT datum \cite{nlg1} of
the B-type open-closed topological Landau-Ginzburg theory defined by
$(X,W)$ admits a family of cochain-level realizations. The
differential models proposed in \cite{LG2} were studied from a
mathematical perspective in \cite{nlg1}, to which we refer the reader
for details. They involve a certain $\O(X)$-linear dg-algebra
$(\PV(X),\updelta_W)$ and a certain dg-category $\DF(X,W)$, as well as
cochain-level realizations of the bulk and boundary traces and
bulk-boundary and boundary-bulk maps of the TFT datum. In this
section, we briefly recall the definition of these structures and of
the differential graded disk algebra of \cite{nlg1}. We also recall the
definition of a dg-category $\F(X,W)$ discussed in loc. cit, which
will arise in later sections. Notice that some of these structures can be
defined without assuming compactness of $Z_W$ (and in this paper 
they are considered in this more general situation), though their physics
interpretation is less clear unless one adds that assumption.

\

\subsection{The differential graded bulk algebra}
\label{subsec:PV}

The $\cinf$-module:
\be
\PV(X)\eqdef \cA(X,\wedge TX) \simeq \cA(X)\otimes_\cinf \rGamma_\infty(X,\wedge T X)~~
\ee
carries a natural multiplication which makes it into a unital and
associative $\cinf$-algebra. For any $i=-d,\ldots, 0$ and $j=0,\ldots,
d$, let:
\be
\PV^{i,j}(X)\eqdef \cA^j(X,\wedge^{|i|}TX)~~
\ee
and set $\PV^{i,j}(X)=0$ for $i\not\in \{-d,\ldots, 0\}$ or $j\not\in
\{0,\ldots, d\}$. Then the decomposition:
\be
\PV(X)\eqdef \bigoplus_{i=-d}^0 \bigoplus_{j=0}^d \PV^{i,j}(X)~~
\ee
makes $\PV(X)$ into a unital associative $\Z\times \Z$-graded
$\cinf$-algebra, whose grading is concentrated in bidegrees $(i,j)$
satisfying $i\in \{-d,\ldots, 0\}$ and $j\in \{0,\ldots, d\}$.  The
{\em canonical $\Z$-grading} of $\PV(X)$ is the total grading of this
bigrading:
\be 
\PV^k(X)\eqdef \bigoplus_{i+j=k} \PV^{i,j}(X)~~(k\in \Z)~~,
\ee 
while the {\em canonical $\Z_2$-grading} is the mod 2
reduction of the former:
\beqa
& & \PV^{\hat 0}(X)\eqdef \bigoplus_{k=\ev} \PV^k(X)~~,\nn\\
& & \PV^{\hat 1}(X)\eqdef \bigoplus_{k=\odd} \PV^k(X)\nn~~.
\eeqa
 We trivially extend $\ioda_W$ to a map from $\PV(X)$ to $\PV(X)$
denoted by the same symbol.  Let $\bbpd:= \bbpd_{\wedge
  TX}:\PV(X)\rightarrow \PV(X)$ be the Dolbeault differential of the
holomorphic vector bundle $\wedge TX$.  Then $(\PV(X),\ioda_W,\bbpd)$
is a bicomplex. By definition, the {\em twisted differential}
$\updelta_W:\PV(X)\rightarrow \PV(X)$ is the total differential of
this bicomplex:
\be
\updelta_W\eqdef \bbpd+\ioda_W~~.
\ee 

\begin{Definition}
The {\em twisted Dolbeault algebra of polyvector-valued forms} 
of the LG pair $(X,W )$ is the
supercommutative $\Z$-graded $\O(X)$-linear dg-algebra
$(\PV(X),\updelta_W)$, where $\PV(X)$ is endowed with the canonical
$\Z$-grading. The {\em cohomological twisted Dolbeault algebra} of $(X,W)$ is the
total cohomology algebra:
\be
\HPV(X,W)\eqdef \rH (\PV(X),\updelta_W)~~.
\ee
\end{Definition}

\subsection{Differential graded categories of holomorphic factorizations}
\label{subsec:DF}

\

\

\begin{Definition}
A {\em holomorphic factorization} of $W$ is a pair $a=(E,D)$, where
$E=E^\0\oplus E^\1$ is a holomorphic vector superbundle on $X$ and $D\in
\rGamma(X,End^\1(E))$ is an odd holomorphic section of $End(E)$ which
satisfies $D^2=W\id_E$.
\end{Definition}

\

\noindent Given two holomorphic factorizations $a_1=(E_1,D_1)$ and
$a_2=(E_2,D_2)$ of $W$, the space $\cA(X,Hom(E_1,E_2))\simeq
 \cA(X)\otimes_\cinf\rGamma_\infty(X,Hom(E_1,E_2))$ is $\Z\times
\Z_2$-graded and carries two differentials, namely the {\em Dolbeault
  differential} $\bbpd_{a_1,a_2}:=\bbpd_{Hom(E_1,E_2)}$ and the {\em
  defect differential} $\md_{a_1,a_2}$. The latter is the
$\cinf$-linear endomorphism of $\cA(X,Hom(E_1,E_2))$ determined
by the condition:
\be
\md_{a_1,a_2}(\rho\otimes f)=(-1)^{\rk \rho} \rho \otimes (D_2\circ f)-(-1)^{\rk \rho +\sigma(f)} \rho\otimes (f\circ D_1)
\ee
for all pure rank forms $\rho \in \cA(X)$ and all pure $\Z_2$-degree
elements $f\in \rGamma_\infty(X,Hom(E_1,E_2))$, where $\sigma(f)$
denotes the $\Z_2$-degree of $f$.  The Dolbeault and defect
differentials square to zero and anti-commute, making
$\cA(X,Hom(E_1,E_2))$ into a $\Z\times\Z_2$-graded bicomplex.  The
{\em twisted differential} $\updelta_{a_1,a_2}$ is defined through:
\be
\updelta_{a_1,a_2}\eqdef \bbpd_{a_1,a_2}+\md_{a_1,a_2}~~.
\ee
The {\em total $\Z_2$-grading} of $\cA(X,Hom(E_1,E_2))$ is given by 
the sum of the $\Z_2$-degree with the mod 2 reduction of the $\Z$-degree. 
Notice that $\updelta_{a_1,a_2}$ is odd with respect to this $\Z_2$-grading. 

\

\begin{Definition}
The {\em twisted Dolbeault category of holomorphic factorizations of $W$}
is the $\Z_2$-graded $\O(X)$-linear dg-category $\DF(X,W)$
defined as follows:
\begin{itemize}
\itemsep 0.0em
\item The objects are the holomorphic factorizations of
  $W$.
\item Given two objects $a_1=(E_1,D_1)$ and
  $a_2=(E_2,D_2)$, the module of morphisms from $a_1$ to $a_2$ is:
\be
\Hom_{\DF(X,W)}(a_1,a_2)\eqdef \cA(X,Hom(E_1,E_2))~~,
\ee
endowed with the total $\Z_2$-grading and with the twisted
differential $\updelta_{a_1,a_2}$.
\item The composition of morphisms is given by the wedge product of
  bundle-valued forms.
\end{itemize}
\end{Definition}

\

\noindent Let:
\be
\HDF(X,W)\eqdef \rH(\DF(X,W))
\ee
denote the total cohomology category of $\DF(X,W)$. We will
show in Section \ref{sec:boundary} that $\HDF(X,W)$ admits a simpler
description when $X$ is a Stein manifold, being equivalent with the
category $\HF(X,W)$ defined below.

\

\begin{Definition}
The {\em holomorphic dg-category of holomorphic factorizations} 
is the $\Z_2$-graded $\O(X)$-linear dg-category $\F(X,W)$ defined as
follows:
\begin{itemize}
\itemsep 0.0em
\item The objects are the holomorphic factorizations of $W$.
\item Given two holomorphic factorizations
  $a_1=(E_1,D_1)$ and $a_2=(E_2,D_2)$ of $W$, the module of morphisms from $a_1$ to $a_2$ is: 
\be
\Hom_{\F(X,W)}(a_1,a_2)\eqdef\rGamma(X,Hom(E_1,E_2))~~,
\ee
endowed with the $\Z_2$-grading, with homogeneous components: 
\be
\Hom_{\F(X,W)}^{\kappa}(a_1,a_2)\eqdef \rGamma(X, Hom^{\kappa}(E_1,E_2))~,~~\forall \kappa\in \Z_2
\ee
and with the differentials $\md_{a_1,a_2}$ determined
uniquely by the condition:
\be
\md_{a_1,a_2}(f)\eqdef D_2\circ f-(-1)^{\kappa} f\circ D_1~,~~\forall f\in \rGamma(X, Hom^{\kappa}(E_1,E_2))~,~~\forall \kappa\in \Z_2~~.
\ee
\item The composition of morphisms is the obvious one. 
\end{itemize}
\end{Definition}

\smallskip

\noindent Let:
\be
\HF(X,W)\eqdef \rH(\F(X,W))~~
\ee
denote the total cohomology category of $\F(X,W)$, viewed as a
$\Z_2$-graded $\O(X)$-linear category.

\smallskip

\begin{remark}
When the critical set $Z_W$ is compact, the path integral arguments of
\cite{LG2} (which motivated the rigorous treatment of \cite{nlg1}) show that the 
cohomological category
$\HDF(X,W)$ can be identified with the category of topological
D-branes of the corresponding open-closed topological field theory. In
this case, the cohomology algebra $\HPV(X,W)$ can be identified
with the bulk algebra of the same theory.
\end{remark}

\subsection{The disk algebra of a holomorphic factorization}
\label{subsec:DiskAlgebra}

Fix a holomorphic factorization $a=(E,D)$ of $W$ and set
$\bbpd_a:=\bbpd_{a,a}= \bbpd_{End(E)}$, $\md_a:=\md_{a,a}= [D,\cdot]$
and $\updelta_a:=\updelta_{a,a}=\bbpd_a+\md_a$. Consider the
$\Z_2$-graded unital associative $\cinf$-algebra:
\be
\PV(X,End(E))\eqdef \cA(X,\wedge TX \otimes End(E))\simeq \PV(X){\hat \otimes}_\cinf\rGamma_\infty(X,End(E))~~,
\ee
where ${\hat \otimes}_\cinf$ denotes the graded tensor product and
$\PV(X)$ is endowed with the canonical $\Z_2$-grading. The {\em
twisted disk differential} $\Delta_a$ is the odd
$\O(X)$-linear differential on $\PV(X,End(E))$ defined through:
\be
\Delta_a\eqdef \updelta_W{\hat \otimes}_{\cA(X)} \id_{\cA(X,End(E))}+\id_{\PV(X)}{\hat \otimes}_{\cA(X)}\updelta_a~~.
\ee

\begin{Definition}
The {\em differential graded disk algebra of a holomorphic factorization $a=(E,D)$} is the
$\O(X)$-linear $\Z_2$-graded unital dg-algebra
$(\PV(X,End(E)),\Delta_a)$. The {\em cohomological disk algebra
  $\HPV(X,a)$ of $a$} is the total cohomology algebra:
\be
\HPV(X,a)\eqdef \rH(\PV(X,End(E)), \Delta_a)~~.
\ee
\end{Definition}

\noindent Let $\bbpd:= \bbpd_{\wedge TX\otimes
  End(E)}:\PV(X,End(E))\rightarrow \PV(X,End(E))$ be the Dolbeault
differential\footnote{Note that symbol $\bbpd$ will be used to
  simplify notation in three different cases: $\bbpd:= \bbpd_{\wedge
    TX\otimes End(E)}$ for the disk algebra, $\bbpd:= \bbpd_{\wedge
    TX}$ in the algebra of polyvector-valued forms and
  $\bbpd:=\bbpd_{a_1,a_2}:=\bbpd_{\End(E_1,E_2)}$ in Section
  \ref{sec:boundary}.}  of the holomorphic vector bundle $\wedge
TX\otimes End(E)$. Then:
\be
\Delta_a=\bbpd+\ioda_W+\md_a~~,
\ee
where we trivially extended $\ioda_W$ and $\md_a$ to mutually
anti-commuting differentials in $\PV(X,End(E))$ \cite{nlg1}. Notice that
$\bbpd$ anticommutes with $\ioda_W$ and $\md_a$.

\

\section{Analytic models for the cohomological twisted Dolbeault algebra $\HPV(X,W)$}
\label{sec:bulk}

\noindent Let $(X,W)$ be a Landau-Ginzburg pair of dimension $d$. 

\subsection{The generally-valid analytic model} 

Let $\H(\cK_W)$ denote the hypercohomology of the sheaf Koszul complex
\eqref{Koszul}, viewed as a finite complex of sheaves of $\cO_X$-modules
concentrated in non-positive degrees.

\

\begin{Proposition}
There exists a natural isomorphism of $\Z$-graded $\O(X)$-modules: 
\be
\HPV(X,W)\simeq_{\O(X)} \H(\cK_W)~~,
\ee
where $\HPV(X,W)$ is endowed with the canonical $\Z$-grading. Thus: 
\ben
\label{Hbicomplex}
\rH^k(\PV(X),\updelta_W)\simeq_{\O(X)} \H^k(\cK_W)~,~~\forall k\in \{-d,\ldots, d\}~~.
\een
Moreover, we have:
\ben
\label{HJacSpSeq}
\H^k(\cK_W)=\bigoplus_{i+j=k}\bE_\infty^{i,j}~~,
\een
where $\bE_\infty^{i,j}$ is the limit of the spectral sequence $\bE:=(\bE_r^{i,j},\dd_r)_{r\geq 0}$ which
starts with:
\ben
\label{E0}
\bE_0^{i,j}\eqdef \PV^{i,j}(X)= \cA^j(X,\wedge^{|i|} TX)~,~~\dd_0\eqdef\bbpd:=\bbpd_{\wedge TX}  ~,~~( i=-d,\ldots, 0~,~~ j=0,\ldots, d)~.
\een
The zeroth page of this sequence is shown in Figure \ref{diag1}.
\end{Proposition}
\begin{figure}
\label{diag1}
\be
\scalebox{1.1}{
\xymatrixcolsep{1.8pc}\xymatrix{
\ar[r]^-{\ioda_W} \bE_0^{-d,d}\ar[r]^-{\ioda_W} & \bE_0^{-d+1,d} \ar[r]^-{\ioda_W} & \bE_0^{-d+2,d} \ar@{..}[r]^-{}&\bE_0^{0,d}\\
\ar[r]^-{\ioda_W}\bE_0^{-d,2}\ar@{..}[u]^{}& \bE_0^{-d+1,2} \ar[r]^-{\ioda_W}\ar@{..}[u]^{} & \bE_0^{-d+2,2}\ar@{..}[u]^{} \ar@{..}[ru]^{} \ar@{..}[r]^-{}& \bE_0^{0,2} \ar@{..}[u]^{}\\
\ar[r]^-{\ioda_W} \bE_0^{-d,1}\ar[u]^{\bbpd}& \bE_0^{-d+1,1}\ar[u]^{\bbpd}  \ar[r]^-{\ioda_W} & \bE_0^{-d+2,1}\ar[u]^{\bbpd}\ar@{..}[r]^-{}& \bE_0^{0,1} \ar[u]^{\bbpd}\\
\ar[r]^-{\ioda_W} \bE_0^{-d,0}\ar[u]^{\bbpd} & \bE_0^{-d+1,0}\ar[u]^{\bbpd} \ar[r]^-{\ioda_W} & \bE_0^{-d+2,0} \ar[u]^{\bbpd}\ar@{..}[r]^-{}& \bE_0^{0,0} \ar[u]^{\bbpd}}}
\ee
\caption{The zeroth page of the spectral sequence $(\bE_r^{i,j},\dd_r)_{r\geq 0}$. }
\end{figure}
\noindent 
\begin{proof}
To compute the hypercohomology of $\cK_W$, we can use\footnote{This is
  because the sheaf $\cA^k\otimes E$ is fine (and hence soft and
  acyclic) for any holomorphic vector bundle $E$ defined on $X$, where we 
identify $E$ with its locally-free sheaf of holomorphic sections.} the
Dolbeault resolutions of the sheaves $\wedge^k TX$:
\ben
\label{DolbRes}
0\rightarrow \wedge^k TX \stackrel{\bbpd}{\rightarrow} \cA^1\otimes \wedge^k TX 
\stackrel{\bbpd}{\rightarrow} \ldots \stackrel{\bbpd}{\rightarrow} \cA^d\otimes \wedge^k TX\rightarrow 0~~,
\een
where $\cA^j$ are the sheaves of smooth $(0,j)$-forms on $X$ (viewed
as sheaves of $\cO_X$-modules by restriction of scalars) and the
tensor product is taken over $\cO_X$. Notice that
$(\PV(X),\ioda_W,\bbpd)$ coincides with the $\Z\times \Z$-graded
bicomplex whose node $(i,j)$ is given by $\cA^j\otimes \wedge^{|i|}
TX$ and whose horizontal and vertical differentials are given
respectively by $\ioda_W$ and $\bbpd$. This
provides an acyclic resolution of the sheaf Koszul complex $\cK_W$ defined
in \eqref{Koszul}. The hypercohomology of $\cK_W$ coincides with the
total cohomology of this bicomplex, which in turn coincides with
$\HPV(X,W)$. Thus \eqref{Hbicomplex} holds. The spectral sequence
$(\bE_r^{i,j},\dd_r)_{r\geq 0}$ is the spectral sequence
determined by the bicomplex described above, showing that
\eqref{HJacSpSeq} holds.~\qed
\end{proof}

\subsection{Analytic models of $\HPV(X,W)$ for the Stein case}

\noindent We refer the reader to Appendix \ref{app:Stein} for the
relevant properties of Stein manifolds and in particular for Cartan's
theorem B (see Theorem \ref{thm:B}), which is crucial for what follows.

\

\begin{Theorem} 
\label{thm:BulkDeg}
Suppose that $X$ is Stein. Then the spectral sequence $\bE$ defined above
collapses at the second page $\bE_2$ and $\HPV(X,W)$ is concentrated in non-positive
degrees. For all $k=-d,\ldots, 0$, the $\O(X)$-module $\HPV^k(X,W)\eqdef\rH^k(\PV(X),\updelta_W)$ is
isomorphic with the cohomology at position $k$ of the following
sequence of finitely-generated projective $\O(X)$-modules:
\ben
\label{H0seq}
(\bK_W):~~0~ {\to} \rH^0(\wedge^d TX)\stackrel{\ioda_W}{\to} \rH^0(\wedge^{d-1} TX)
\stackrel{\ioda_W}{\to} \ldots \stackrel{\ioda_W}{\to} \rH^0(TX) \stackrel{\ioda_W}{\to} \O(X)\rightarrow 0~~,
\een
where $\O(X)$ sits in position zero. 
\end{Theorem}

\

\begin{proof}
Since $X$ is Stein, Cartan's theorem B implies
$\bE_1^{i,j}=\rH^j_{\bbpd}(\cA(X,\wedge^{|i|} TX))=0$ for $j>0$ and
all $i=-d,\ldots ,0$. Thus the only non-trivial row of the page
$\bE_1$ of the spectral sequence is the bottom row
$\bE_1^{\bullet,0}$ with differential $\dd_1\eqdef\ioda_W$, whose nodes are given 
by (see Figure \ref{diag2}):
\be
\bE_1^{i,0}= \rH^0_{\bbpd}(\cA(X,\wedge^{|i|} TX))=\rH_{\bbpd}(\PV^{i,0}(X))=\rGamma(X,\wedge^{|i|} TX)=\rH^0(\wedge^{|i|} TX)~~
\ee
for all $i=-d,\ldots, 0$. Thus page $\bE_1$ reduces to 
the form shown in Figure \ref{diag2}.
\begin{figure}
\label{diag2}
\be
\scalebox{1.0}{
\xymatrixcolsep{1.8pc}
\xymatrix{
\bE_1^{-d,d}\!=\!0~\ar[r]^-{} &~ \bE_1^{-d+1,d}\!=\!0~ \ar[r]^-{} & ~\bE_1^{-d+2,d}\!=\!0 ~\ar@{..}[r]^-{}&~\bE_1^{0,d}\!=\!0\\
\bE_1^{-d,2}\!=\!0~\ar@{..}[u]^{} \ar[r]^-{}& ~\bE_1^{-d+1,2}\!=\!0~ \ar[r]^-{}\ar@{..}[u]^{} &~ \bE_1^{-d+2,2}\!=\!0~\ar@{..}[u]^{} \ar@{..}[ru]^{} \ar@{..}[r]^-{}&~ \bE_1^{0,2}\!=\!0 \ar@{..}[u]^{}\\
\bE_1^{-d,1}\!=\!0~\ar[u]^{\bbpd}\ar[r]^-{}& ~\bE_1^{-d+1,1}\!=\!0~ \ar[u]^{\bbpd}  \ar[r]^-{} & ~\bE_1^{-d+2,1}\!=\!0~\ar[u]^{\bbpd}\ar@{..}[r]^-{}&~ \bE_1^{0,1}\!=\!0 \ar[u]^{\bbpd}\\
\bE_1^{-d,0}~\ar[u]^{\bbpd} \ar[r]^-{\ioda_W} & ~\bE_1^{-d+1,0}~\ar[u]^{\bbpd} \ar[r]^-{\ioda_W} & ~\bE_1^{-d+2,0} ~\ar[u]^{\bbpd}\ar@{..}[r]^-{}& ~\bE_1^{0,0} \ar[u]^{\bbpd}
}}
\ee
\caption{The first page of the spectral sequence $(\bE_r^{i,j},\dd_r)_{r\geq 0}$. }
\end{figure}
We note that all differentials $\dd_2$ of
$\bE_2^{\bullet,\bullet}$ will be trivial, since they are maps between
different rows. Hence the spectral sequence collapses at $\bE_2$ and
we have
$\bE_\infty^{k}=\bE_2^{k,0}=\rH^k_{\ioda_W}(\bE_1^{\bullet,0})=\rH^k(\bK_W)$
for all $k=-d,\ldots,0$. Since $\wedge^k TX$ are vector bundles, the
Serre-Swan theorem for Stein manifolds (see \cite{ForsterStein,Morye})
implies that \eqref{H0seq} is a sequence of finitely-generated
projective $\O(X)$-modules.~~\qed
\end{proof}

\

\begin{Proposition}
\label{prop:HJac}
Suppose that $X$ is Stein and $\dim_\C Z_W=0$. Then $\HPV^k(X,W)=0$ for
$k\neq 0$ and there exists a natural isomorphism of $\O(X)$-modules:
\beqan
\label{HPVStein}
\HPV^0(X,W)\simeq_{\O(X)} \rH^0(Jac_W)=\Jac(X,W)~~\nn~~.
\eeqan
Moreover, we have an isomorphism of $\O(X)$-algebras:
\ben
\label{JacAlgebraStein}
\Jac(X,W)\simeq \O(X)/\J(X,W)~~.
\een
\end{Proposition}

\begin{proof}
Lemma \ref{prop:Koszul} implies that the hypercohomology $\H^k(\cK_W)$
of the Koszul complex \eqref{Koszul} coincides with the sheaf cohomology
of the Jacobi sheaf $Jac_W$:
\ben
\H^k(\cK_W)=\rH^k(Jac_W)~~.
\een
Combining this with \eqref{Hbicomplex} gives: 
\ben
\label{HJac}
\HPV^k(X,W)\simeq_{\O(X)}\rH^k(Jac_W)~,~~\forall k \in \Z~~.
\een
Thus $\HPV^k(X,W)=0$ for $k\neq 0$ by Cartan's Theorem B (see Appendix
\ref{app:Stein}) and $\HPV^0(X,W)\simeq_{\O(X)}
\rH^0(Jac_W)=\Jac(X,W)$.  Since $X$ is Stein, Theorem
\ref{thm:BulkDeg} shows that $\HPV^0(X,W)$ also coincides with the
cohomology of the sequence \eqref{H0seq} at position zero, which
equals $\O(X)/\im(\ioda_W:\rH^0(TX)\rightarrow
\O(X))=\O(X)/\J(X,W)$. This shows that \eqref{JacAlgebraStein}
holds. ~\qed
\end{proof}

\begin{remark}
In applications to physics, the set $Z_W$ is usually assumed to be
compact, since this condition implies \cite{nlg1} finite-dimensionality
of $\HPV(X,W)$ as well as Hom-finiteness of the category $\HDF(X,W)$
over $\C$. In such applications, the condition $\dim_\C Z_W=0$ is
automatically satisfied when $X$ is Stein, provided that the critical set is compact (see Remark \ref{rem:SteinFinite}).
\end{remark}

\subsection{Analytic model of $\HPV(X,W)$ when $X$ is Stein and holomorphically parallelizable with $\dim_\C Z_W=0$}

Recall that $X$ is called {\em holomorphically parallelizable} if its
holomorphic tangent bundle $TX$ is holomorphically trivial. This
happens iff $X$ admits global holomorphic frames\,\footnote{A global
  holomorphic frame of $X$ is a family of globally-defined holomorphic
  vector fields $u_1,\ldots, u_d\in \rGamma(X,TX)$ such that the
  vectors $u_1(x),\ldots, u_d(x)$ form a basis of $T_xX$ for any point
  $x\in X$.}.

\

\begin{Lemma}
\label{lemma:par}
Suppose that $X$ is holomorphically parallelizable and let
$u_1,\ldots, u_d$ be any globally-defined holomorphic frame of
$TX$. Then the critical sheaf $\cJ_W$ is generated on every open
subset of $X$ by the restrictions of the holomorphic functions
$u_1(W),\ldots, u_d(W)\in \O(X)$. In particular, the critical ideal
$\J(X,W)$ is generated by these holomorphic functions:
\be
\J(X,W)=\langle u_1(W), \ldots, u_d(W) \rangle~~.
\ee
\end{Lemma}

\begin{proof}
The holomorphic frame $u_1,\ldots, u_d$ of $X$ induces an isomorphism of
locally-free sheaves:
\be
\psi:\cO_X^{\oplus d}\stackrel{\sim}{\rightarrow} TX
\ee 
given on open subsets $U\subset X$ by:
\be
\psi_U(f_1,\ldots, f_d)\eqdef \sum_{i=1}^d f_i u_i\in \rGamma(U,TX)~,~~\forall f_1,\ldots, f_d\in \cO_X(U)~~.
\ee
Since $\ioda_W(u_i)= -\i u_i(W)\in \O(X)$, the morphism of sheaves
$\ioda_W:TX\rightarrow \cO_X$ can be identified with the morphism
$\tau:\cO_X^{\oplus d}\rightarrow \cO_X$ which is given on open
sets $U$ by:
\be
\tau_U(f_1,\ldots, f_d)\eqdef -\i \sum_{i=1}^d f_i ~[u_i(W)]|_U\in \cO_X(U)~,~~\forall f_1,\ldots, f_d\in \cO_X(U)~~.
\ee
This implies $\cJ_W\!=\!\im \tau\!=\!\sum_{i=1}^d\! \cO_X u_i(W)$ and in
particular $\J(X,W)\!=\!\cJ_W(X)\!= \!\sum_{i=1}^d\! \O(X) u_i(W)$. Thus
$\J(X,W)$ coincides with the ideal $\langle u_1(W),\ldots,
u_d(W)\rangle$ generated inside $\O(X)$ by the holomorphic functions
$u_i(W)$.~\qed
\end{proof}

\

\begin{Proposition} 
\label{prop:BulkPar}
Let $(X,W)$ be a Landau-Ginzburg pair such that $X$ is Stein and
holomorphically parallelizable and such that $\dim_\C Z_W=0$. Let
$u_1,\ldots, u_d$ be any holomorphic frame of $X$. Then
$\HPV(X,W)=\HPV^0(X,W)\simeq_{\O(X)} \Jac(X,W)$ and we have:
\be
\Jac(X,W)=\O(X)/\langle u_1(W),\ldots, u_d(W)\rangle~~.
\ee
\end{Proposition}
 
\begin{proof}
Follows immediately from Lemma \ref{lemma:par} and Proposition \ref{prop:HJac}. ~\qed
\end{proof}

\section{Analytic models for the category $\HDF(X,W)$}
\label{sec:boundary}

\noindent Let $(X,W)$ be a Landau-Ginzburg pair of dimension $d$. 

\subsection{A generally-valid analytic model}

Let $a_1\eqdef (E_1,D_1)$ and $a_2\eqdef (E_2,D_2)$ be two holomorphic
factorizations of $W$. Let $\bbpd:= \bbpd_{a_1,a_2}$ and $\md:=
\md_{a_1,a_2}$ be the Dolbeault and defect differentials on
$\Hom_{\DF(X,W)}(a_1,a_2)=\cA(X,Hom(E_1,E_2))$. Let
$\updelta:=\updelta_{a_1,a_2}$ be the twisted differential. Consider
the complex $(\Hom_{\DF(X,W)}(a_1,a_2),\updelta)$ (endowed with the
total $\Z_2$-grading), whose total cohomology equals the $\Z_2$-graded
$\O(X)$-module $\Hom_{\HDF(X,W)}(a_1,a_2)$.

\

\begin{Definition}
\label{Def_unwin}
The {\em unwinding} of the $\Z\times \Z_2$-graded bicomplex
$(\cA(X,Hom(E_1,E_2)),\bbpd,\md)$ is the $\Z\times \Z$-graded complex
$\cC^{\bullet,\bullet}$ of $\O(X)$-modules with homogeneous components:
\be
\cC^{i,j} \eqdef \cA^j(X, Hom^{{\hat i}}(E_1,E_2))~,~~\forall i, j\in \Z~~,
\ee
whose non-trivial horizontal and vertical differentials are $\md$ and
$\bbpd$, both of which have degree $+1$ (see Diagram 3). This
bicomplex is 2-periodic in the horizontal direction and concentrated
in degrees $0,\ldots, d$ in the vertical direction, thus $\cC^{i,j}=0$
for $j\not \in \{0,\ldots, d\}$ and all $i$.
\end{Definition}
\begin{figure}
\label{diag3}
\be
\scalebox{1.1}{
\xymatrixcolsep{1.8pc} \xymatrix{
\ar@{..}[r]^-{} & \cC^{-2,d} \ar[r]^-{\md}& \cC^{-1,d} \ar[r]^-{\md} &  \ar[r]^-{\md}\cC^{0,d}& \cC^{1,d} \ar[r]^-{\md} & \cC^{2,d} \ar@{..}[r]^-{}&\\
\ar@{..}[r]^-{} & \cC^{-2,2}\ar@{..}[u]^{} \ar[r]^-{\md}& \cC^{-1,2} \ar[r]^-{\md}\ar@{..}[u]^{} &  \ar[r]^-{\md}\cC^{0,2}\ar@{..}[u]^{}& \cC^{1,2} \ar[r]^-{\md}\ar@{..}[u]^{} & \cC^{2,2}\ar@{..}[u]^{} \ar@{..}[r]^-{}&\\
\ar@{..}[r]^-{} & \ar[r]^-{\md}\cC^{-2,1}\ar[u]^{\bbpd}& \cC^{-1,1}\ar[u]^{\bbpd}  \ar[r]^-{\md} & \ar[r]^-{\md}\cC^{0,1}\ar[u]^{\bbpd}& \cC^{1,1}\ar[u]^{\bbpd}  \ar[r]^-{\md} & \cC^{2,1}\ar[u]^{\bbpd}\ar@{..}[r]^-{}&\\
\ar@{..}[r]^-{} & \ar[r]^-{\md} \cC^{-2,0}\ar[u]^{\bbpd} & \cC^{-1,0}\ar[u]^{\bbpd} \ar[r]^-{\md} & \ar[r]^-{\md} \cC^{0,0}\ar[u]^{\bbpd} & \cC^{1,0}\ar[u]^{\bbpd} \ar[r]^-{\md} & \cC^{2,0} \ar[u]^{\bbpd}\ar@{..}[r]^-{}&}
}
\ee 
\caption{The graded bicomplex $(\cC^{i,j}, \md,\bbpd)$.}
\end{figure}

Consider the total complex $\cC^{\bullet}$ of the bicomplex
$\cC^{\bullet,\bullet}$. Its homogeneous components are:
\be
\cC^k \eqdef \bigoplus_{i+j=k}\cC^{i,j}~,~~\forall k\in \Z~~
\ee
and its differential equals $\updelta=\bbpd+\md$. Note that the sum on the right
hand side has only a finite number of nonzero terms. The complex $\cC^{\bullet}$
is $2$-periodic:
\be
\cC^k=\cC^{k+2}~,~~\forall k\in \Z
\ee
and the same holds for its cohomology $\rH^k(\cC^{\bullet})$:
\be
\rH^k(\cC^{\bullet})=\rH^{k+2}(\cC^{\bullet})~,~~\forall k\in \Z~~.
\ee
We have:
\ben
\label{HC}
\rH^k(\cC)=\rH^{{\hat k}}(\cA(X,Hom(E_1,E_2)),\updelta)=\Hom_{\HDF(X,W)}^{\hat k}(a_1,a_2)~,~~\forall k\in \Z~~.
\een
Consider the  spectral sequence  $\bE=(\bE_r^{\bullet,\bullet},{\dd}_r)_{r\geq 0}$ 
whose zeroth page is given by:
\be
\label{ss_bndr}
\bE_0^{i, j}=\cC^{i,j}=\cA^j(X, Hom^{{\hat i}}(E_1,E_2))
\ee
endowed with the vertical differential ${\dd}_0\eqdef\bbpd$ shown in Figure \ref{diag3}.

\

\begin{Proposition}
\label{prop:Boundary}
The spectral sequence $\bE$ defined above converges. Moreover, we have
a natural isomorphism of $\O(X)$-modules:
\be
\bigoplus_{i+j=t}{\bE}_{d+2}^{i,j}\simeq_{\O(X)} \Hom_{\HDF(X,W)}^{\hat t}(a_1,a_2)~,~~\forall t\in \Z~~.
\ee
\end{Proposition}

\begin{proof}
The spectral sequence $\bE$ is given by the following
(horizontal)
filtration $F^p$ of the bicomplex $\cC^{\bullet,\bullet}$:
\be
F^p:\cC^{\bullet,\bullet}\rightarrow \cC^{\bullet,\bullet}~,~~ \cC^{i,j}\mapsto 
\begin{cases} \cC^{i,j}~, ~~ \text{if} ~~ i\geq p ~,\\
 0~~, ~~~~ \text{otherwise} ~~.\end{cases}
\ee
This is an infinite filtration on the horizontally unbounded complex,
but the induced filtration $F^p(\cC^\bullet)$ on the total complex is
finite.  Indeed, we have $F^{k-d}(\cC^{k})=\cC^k$ and
$F^{k+1}(\cC^k)=0$ for any $k$ since $\cC^{i,j}$ vanishes for $j<0$ or
$j>d$. This ensures that the spectral sequence converges to the
total cohomology $\rH(\cC^{\bullet})$ 
of the 2-periodic complex $\cC^{\bullet}$
(see \cite[$\SSS$14]{BoTu}). Moreover, it
degenerates at page $d+2$. Thus:
\be
\bigoplus_{i+j=t}{\bE}_\infty^{i,j}=\bigoplus_{i+j=t}{\bE}_{d+2}^{i,j}=\rH^t(\cC^\bullet)~~.
\ee
Combining this with \eqref{HC} gives the conclusion. ~\qed
\end{proof}

\smallskip

\begin{remark}
For a pair of holomorphic factorizations  $(a_1,a_2)$ the limit of the spectral sequence 
$\bE$ defined above looks like the "hypercohomology" of the following  2-periodic
complex of locally-free sheaves:
\ben
\label{DefComplex}
(\cQ_{a_1,a_2}):~ ~\ldots \rightarrow Hom^\1(E_1, E_2)\stackrel{\md}{\rightarrow}
 Hom^\0(E_1, E_2) \stackrel{\md}{\rightarrow} Hom^\1(E_1,E_2)\rightarrow \ldots~~,
\een
where $Hom^\0(E_1,E_2)$ sits in even positions.  Indeed, the columns give (2-periodic) 
bounded acyclic (Dolbeault) resolutions of nodes $\Hom^\0(E_1,E_2)$ and $\Hom^\1(E_1,E_2)$, 
while $\Hom_{\HDF(X,W)}(a_1,a_2)$ is isomorphic to the total cohomology of 
such bicomplex. However, the notion of hypercohomology for unbounded complexes 
is ambiguous (see \cite{Wei}).
\end{remark}

\subsection{An analytic model of $\HDF(X,W)$ when $X$ is Stein}

Recall the category $\HF(X,W)$ defined in Subsection \ref{subsec:DF}.

\

\begin{Lemma}
\label{lemma:BoundaryStein}
Suppose that $X$ is Stein. Then the spectral sequence $\bE$ defined in (\ref{ss_bndr}) 
degenerates at $\bE_2$ and we have an isomorphism of $\Z_2$-graded
$\O(X)$-modules:
\ben
\label{HFStein}
\Hom_{\HDF(X,W)}(a_1,a_2)\simeq_{\O(X)} \Hom_{\HF(X,W)}(a_1,a_2)~~.
\een
\end{Lemma}

\begin{proof}
The first page of the spectral sequence is given by:
\be
\bE_1^{i,j} \eqdef \rH(\bE_0^{i,j},\bbpd)=\rH(\cA^j(X, Hom^{\hat i}(E_1,E_2)),
\bbpd)=\rH^j_{\bbpd}(Hom^{\hat i}(E_1,E_2))~~.
\ee 
Since $X$ is Stein, Cartan's Theorem B yields:
\be
\bE_1^{i,j} = 0 ~~\mathrm{for}~j>0~\text{and}~~i\in \Z~~.
\ee
Hence all differentials ${\dd}_2$ of the next page $\bE_2$
must vanish, since they are maps between different rows. 
Thus the spectral sequence degenerates at $\bE_2$ and we have:
\ben
\label{HFSp}
\Hom_{\HDF(X,W)}^{\hat t}(a_1,a_2)\simeq_{\O(X)}\bigoplus_{i+j=t} {\bE}_2^{i,j}=
~\bigoplus_{i+j=t} \rH(\bE_1^{i,j},{\dd}_1)~,~~\forall t\in \Z~~.
\een
The only nonzero row of $\bE_1$ is the bottom row $\bE_1^{\bullet,0}$,
which is a 2-periodic sequence with nodes $\bE_1^{i,0}=\rH^0_{\bbpd}(Hom^{\hat
  i}(E_1,E_2))=\rGamma(X,Hom^{\hat i}(E_1,E_2))$ and differential
${\dd}_1\eqdef\md$:
\be
\label{d48}
\ldots \rightarrow{\bE}^{i-1,0}_1
\mathrel{\mathop{\longrightarrow}^{\mathrm{\md}}} {\bE}^{i,0}_1
\mathrel{\mathop{\longrightarrow}^{\mathrm{\md}}} {\bE}^{i+1,0}_1
\rightarrow \ldots
\ee
The cohomology of this sequence at node $i$ equals
$\rH_{\md}(\rGamma(X,Hom^{\hat i}(E_1,E_2)))=\Hom_{\HF(X,W)}^{\hat
  i}(a_1,a_2)$. Hence \eqref{HFSp} reduces to:
\be
\Hom_{\HDF(X,W)}^{\hat t}(a_1,a_2)\simeq_{\O(X)}\rH(\bE_1^{t,0},\md)
=\rH_{\md}^{\hat t}(\rGamma(X,Hom(E_1,E_2)))=\Hom_{\HF(X,W)}^{\hat
  t}(a_1,a_2)~~,
\ee
for all $t\in \Z$, which gives \eqref{HFStein}.~\qed
\end{proof}

\

\begin{Theorem}
\label{thm:BoundaryStein}
Suppose that $X$ is Stein. Then $\HDF(X,W)$ and $\HF(X,W)$ are
equivalent as $\Z_2$-graded $\O(X)$-linear categories.
\end{Theorem}

\begin{proof}
It is easy to see that the isomorphism of Lemma
\ref{lemma:BoundaryStein} is natural with respect to $a_1$ and $a_2$
and that it preserves units. ~\qed
\end{proof}

\subsection{Relation to projective analytic factorizations in the Stein case}

Recall that $\O(X)=\cO_X(X)$ denotes the commutative ring of
complex-valued holomorphic functions defined on $X$.

\

\begin{Definition}
An {\em $\O(X)$-supermodule} is a $\Z_2$-graded
$\O(X)$-module $M$ endowed with a direct sum decomposition 
$M=M^\0\oplus M^\1$ into submodules. 
\end{Definition}

\

\noindent Notice that $\O(X)$-supermodules form an $\O(X)$-linear
$\Z_2$-graded category $\Mod^s_{\O(X)}$ if we define the Hom space
$\Hom_{\O(X)}(M_1,M_2)$ from a supermodule $M_1$ to a supermodule
$M_2$ to be the $\Z_2$-graded $\O(X)$-module with homogeneous
components:
\beqan
\label{ModGrading}
& & \Hom_{\O(X)}^\0(M_1,M_2)\eqdef \Hom_{\O(X)}(M_1^\0,M_2^\0)\oplus \Hom_{\O(X)}(M_1^\1,M_2^\1)~~,\nn\\ 
& & \Hom_{\O(X)}^\1(M_1,M_2)\eqdef \Hom_{\O(X)}(M_1^\0,M_2^\1)\oplus \Hom_{\O(X)}(M_1^\1,M_2^\0)~~.
\eeqan
The composition is defined in the obvious manner. Given an
$\O(X)$-supermodule $M$, let
\be
\End_{\O(X)}(M)\eqdef \Hom_{\O(X)}(M,M)~.
\ee

\begin{Definition}
An $\O(X)$-supermodule $M=M^\0\oplus M^\1$ is called {\em
  finitely-generated} if both of its $\Z_2$-homogeneous components
$M^\0$ and $M^\1$ are finitely-generated over $\O(X)$. It is
called {\em projective} if both $M^\0$ and $M^\1$ are projective
$\O(X)$-modules.
\end{Definition}

\

\noindent Finitely-generated $\O(X)$-supermodules form a full
$\Z_2$-graded $\O(X)$-linear subcategory $\mod_{\O(X)}^s$ of
$\Mod_{\O(X)}^s$, while projective and finitely-generated
$\O(X)$-supermodules form a full $\Z_2$-graded $\O(X)$-linear
subcategory $\proj^{s}_{\O(X)}$ of $\mod^s_{\O(X)}$.

\

\begin{Definition}
A {\em projective analytic factorization} of $W$ is a pair $(P,D)$,
where $P$ is a finitely-generated projective $\O(X)$-supermodule and
$D\in \End_{\O(X)}^\1(P)$ is an odd endomorphism of $P$ such that
$D^2=W\id_P$.
\end{Definition}

\

\begin{Definition}
The {\em dg-category $\PF(X,W)$ of projective analytic factorizations}
of $W$ is the $\Z_2$-graded  $\O(X)$-linear dg-category defined as
follows:
\begin{itemize}
\item The objects are the projective analytic factorizations of $W$.
\item Given two projective analytic factorizations $(P_1,D_1)$ and
  $(P_2,D_2)$ of $W$, we set: 
\be
\Hom_{\PF(X,W)}((P_1,D_1),(P_2,D_2))\eqdef
  \Hom_{\O(X)}(P_1,P_2)~~, 
\ee
endowed with the $\Z_2$-grading \eqref{ModGrading} inherited from
$\mod^s_{\O(X)}$ and with the $\O(X)$-linear odd differential
$\md:=\md_{(P_1,D_1),(P_2,D_2)}$ determined uniquely by the
condition:
\be
\md(f)\eqdef D_2\circ f-(-1)^{\deg f}f\circ D_1 
\ee 
for all elements $f\in \Hom_{\O(X)}(P_1,P_2)$ which have pure
$\Z_2$-degree.
\item The composition of morphisms is inherited from $\mod^s_{\O(X)}$.
\end{itemize}
The {\em cohomological category $\HPF(X,W)$ of analytic projective
  factorizations} of $W$ is the total cohomology category:
\be
\HPF(X,W)\eqdef \rH(\PF(X,W))~~,
\ee
which is a $\Z_2$-graded $\O(X)$-linear category.
\end{Definition}

\

\noindent Let us assume that $X$ is Stein. Then the Serre-Swan theorem
for Stein manifolds (see \cite{ForsterStein,Morye}) states that the functor
$\rGamma_X\eqdef \rGamma(X,\cdot)$ of taking {\em global} holomorphic
sections gives an equivalence of $\O(X)$-linear categories:
\beqa
\rGamma_X:\VB(X)\stackrel{\sim}{\rightarrow} \proj_{\O(X)}~~,
\eeqa
where $\proj_{\O(X)}$ is the category of finitely-generated projective
$\O(X)$-modules. This induces a ($\O(X)$-linear, degree zero) 
dg-functor $\mGamma_X:\F(X,W)\rightarrow \PF(X,W)$ which sends a
holomorphic factorization $(E,D)$ of $W$ into the projective
factorization $\mGamma_X(E,D)\eqdef (\rGamma(X,E),D)$, where 
$D\in \rGamma(X,End^\1(E))\simeq_{\O(X)} \End_{\O(X)}^\1(\rGamma(X,E))$. 

\

\noindent The proof of the following statement is immediate:

\

\begin{Proposition}
\label{prop:proj}
Assume that $X$ is Stein. Then the dg-functor $\mGamma_X$ is an
equivalence of $\Z_2$-graded $\O(X)$-linear dg-categories between
$\F(X,W)$ and $\PF(X,W)$. In particular, the $\Z_2$-graded
$\O(X)$-linear cohomological categories $\HF(X,W)$ and $\HPF(X,W)$ are
equivalent.
\end{Proposition}

\

\noindent Theorem \ref{thm:BoundaryStein} and Proposition \ref{prop:proj} imply 
that the categories $\HDF(X,W)$ and $\HPF(X,W)$ are equivalent when $X$ is Stein. 

\subsection{Free holomorphic factorizations and analytic matrix factorizations}

\

\

\begin{Definition}
A holomorphic vector bundle $E$ on $X$ with $\rk_\C E=r$ is called
{\em holomorphically trivial} if it is isomorphic (as a holomorphic
vector bundle) with the trivial holomorphic vector bundle
$\cO_X^{\oplus r}$. A holomorphic vector superbundle $E=E^\0\oplus
E^\1$ on $X$ is called {\em holomorphically trivial} if both its even
and odd sub-bundles $E^\0$ and $E^\1$ are holomorphically trivial.
\end{Definition}

\

\begin{remark}
Two holomorphic vector bundles $E$ and $F$ defined on $X$ are
isomorphic in the category $\VB(X)$ iff they are isomorphic in the
usual category of holomorphic vector bundles defined on $X$. Indeed,
$\VB(X)$ is equivalent with the full subcategory of $\Coh(X)$
consisting of locally-free sheaves of finite rank. Also, an isomorphism
of sheaves between the sheaves of holomorphic sections $\cE$ and $\cF$
of $E$ and $F$ has trivial kernel and image equal to $\cF$, which
means that the corresponding isomorphism in the category $\VB(X)$ is
an ordinary isomorphism of vector bundles (since its kernel and image
are sub-bundles of $E$ and $F$, respectively). In particular, a vector
bundle $E$ is holomorphically trivial iff it is isomorphic with a
trivial vector bundle in the category $\VB(X)$.
\end{remark}

\smallskip

\noindent Let $\VB_\triv(X)$ denote the full subcategory of $\VB(X)$ whose
objects are the holomorphically trivial holomorphic vector bundles
defined on $X$ and $\VB_\triv^s(X)$ denote the full subcategory of
$\VB^s(X)$ whose objects are the holomorphically trivial holomorphic
vector superbundles defined on $X$.

\smallskip

\begin{remark}
Suppose that $X$ is Stein. Then the Oka-Grauert principle
(see \cite{OkaGrauert,Oka}) implies that a holomorphic vector bundle $E$ is
holomorphically trivial iff it is topologically trivial,
i.e. isomorphic with a trivial vector bundle in the category of
complex vector bundles defined on $X$.
\end{remark}

\smallskip

\begin{Definition}
A holomorphic factorization $(E,D)$ of $W$ is called {\em free} if the
holomorphic vector superbundle $E$ is holomorphically trivial.
\end{Definition}

\

\begin{Definition} 
The {\em category $\F_\free(X,W)$ of free holomorphic factorizations
  of $W$} is the full $\Z_2$-graded $\O(X)$-linear dg-subcategory of
the category $\F(X,W)$ whose objects are the free holomorphic
factorizations of $W$. The {\em cohomological category $\HF_\free(X,W)$
  of free holomorphic factorizations of $W$} is the $\Z_2$-graded
$\O(X)$-linear category defined as the total cohomology category of
$\F_\free(X,W)$:
\be
\HF_\free(X,W)\eqdef \rH(\F_\free(X,W))~.
\ee
\end{Definition}

\begin{Definition}
An $\O(X)$ supermodule $M=M^\0\oplus M^\1$ is called {\em free} if
its even and odd submodules $M^\0$ and $M^\1$ are free $\O(X)$-modules. 
\end{Definition}

\

\noindent Let $\free_{\O(X)}$ denote the full subcategory of
$\proj_{\O(X)}$ consisting of those finitely-generated projective $\O(X)$-modules
which are free. Let $\free^s_{\O(X)}$ denote the full subcategory of
$\proj^s_{\O(X)}$ consisting of those finitely-generated projective 
$\O(X)$-supermodules which are free.

\

\begin{Definition}
An {\em analytic matrix factorization} of\, $W$ is a pair $(M,D)$, where
$M$ is a free and finitely-generated projective $\O(X)$-supermodule and
$D\in \End^\1_{\O(X)}(M)$ is an odd endomorphism of $M$ such that
$D^2=W\id_M$.
\end{Definition}

\

\begin{Definition}
The {\em dg-category $\MF(X,W)$ of analytic matrix factorizations} of
$W$ is the full $\Z_2$-graded $\O(X)$-linear dg-subcategory of
$\PF(X,W)$ whose objects are the analytic matrix factorizations of
$W$. The {\em cohomological category $\HMF(X,W)$ of analytic matrix
  factorizations} is the $\Z_2$-graded $\O(X)$-linear category defined
as the total cohomology category of $\MF(X,W)$:
\be
\HMF(X,W)\eqdef \rH(\MF(X,W))~.
\ee
\end{Definition}

\noindent When $X$ is Stein, the Serre-Swan equivalence
$\rGamma_X:\VB(X)\stackrel{\sim}{\rightarrow} \proj_{\O(X)}$ restricts to
an equivalence of $\O(X)$-linear categories between $\VB_\triv(X)$ and
$\free_{\O(X)}$. This implies:

\

\begin{Proposition}
Assume that $X$ is Stein. Then the equivalence of categories
$\mGamma_X:\F(X,W)\rightarrow \PF(X,W)$ restricts to an equivalence of
$\Z_2$-graded $\O(X)$-linear categories between $\F_\free(X,W)$ and
$\MF(X,W)$. In particular, the $\Z_2$-graded $\O(X)$-linear categories
$\HF_\free(X,W)$ and $\HMF(X,W)$ are equivalent.
\end{Proposition}

\

\begin{Corollary}
\label{cor:MF}
Assume that $X$ is Stein and that any holomorphic vector bundle
defined on $X$ is topologically trivial. Then $\F(X,W)=\F_\free(X,W)$
and the $\Z_2$-graded $\O(X)$-linear dg-categories $\F(X,W)$ and
$\MF(X,W)$ are equivalent. In particular, the $\Z_2$-graded
$\O(X)$-linear categories $\HF(X,W)$ and $\HMF(X,W)$ are equivalent.
\end{Corollary}

\begin{proof}
By the Oka-Grauert principle (see \cite{OkaGrauert,Oka}), topological
triviality of a holomorphic vector bundle $E$ implies holomorphic
triviality of $E$. Thus the hypothesis implies
$\F(X,W)=\F_\free(X,W)$. The remaining statements follow immediately
from the results above.~\qed
\end{proof}

\smallskip

\begin{remark}
In general, the category $\HF(X,W)$ has many more objects than the
category $\HF_\free(X,W)$, since a generic Calabi-Yau Stein manifold
$X$ has many holomorphic vector bundles which are not topologically
trivial. 
\end{remark}

\subsection{Topologically non-trivial elementary holomorphic factorizations}
\label{subsec:nontriv}

\

\

\begin{Definition}
Let $(X,W)$ be an LG pair where $X$ is a Stein manifold. 
A holomorphic factorization $(E,D)$ of $W$ is called {\em elementary} 
if $\rk E^\0=\rk E^\1=1$. 
\end{Definition}

\

\noindent Topologically non-trivial elementary holomorphic
factorizations can be constructed as follows on any Calabi-Yau Stein
manifold whose second cohomology with integer coefficients is
non-trivial. Let $X$ be a Calabi-Yau Stein manifold with
$\rH^2(X,\Z)\neq 0$. Then $X$ admits topologically non-trivial complex
line bundles. By the Oka-Grauert principle (see \cite{OkaGrauert}), it
follows that $X$ also admits nontrivial holomorphic line bundles.  By
the same principle, any smooth section of a holomorphic line bundle
can be deformed to a holomorphic section; in particular, any
holomorphic line bundle on $X$ admits nontrivial holomorphic sections.
Given a non-trivial holomorphic line bundle $L$ on $X$, let $v$ be any
nontrivial holomorphic section of $L$ and $u$ be any nontrivial
holomorphic section of the dual line bundle $L^{-1}$. Let $E\eqdef
\cO_X\oplus L$ (viewed as a $\Z_2$-graded holomorphic vector bundle
with even and odd components given respectively by $\cO_X$ and $L$)
and let $D=\left[\begin{array}{cc} 0 &v\\ u & 0\end{array}\right]$ be
the odd section of the $\Z_2$-graded holomorphic vector bundle
$Hom(E,E)\simeq E^\vee \otimes E$ whose component from $\cO_X$ to $L$
is given by $u$ and whose component from $L$ to $\cO_X$ is given by
$v$. Then the tensor product $u\otimes v\simeq v\otimes
u\in\rH^0(L\otimes L^{-1})$ can be identified with a holomorphic
function $W$ defined on $X$ (which is not identically zero) using any
isomorphism between $L\otimes L^{-1}$ and the trivial holomorphic line
bundle $\cO_X$.  Thus $(E,D)$ is a holomorphic factorization of $W$
whose underlying $\Z_2$-graded holomorphic vector bundle $E$ is
non-trivial. This implies that the projective analytic factorization
$(P,D)$ (where $P\eqdef \rH^0(E)$) which corresponds to $(E,D)$ by
Proposition \ref{prop:proj} is not free. Notice that the holomorphic
function $W$ has isolated critical points when the sections $u$
and $v$ are generic.

\

\begin{remark}
Let $\Pic_\an(X)$ denote the analytic Picard group of holomorphic line
bundles on any Stein manifold $X$. Then the assignment $L \to c_1(L)$
induces an isomorphism $\Pic_\an (X) \to \rH^2(X,\Z)$ (see \cite[Lemma
2.5]{picard}).  When $X$ is the analytic space associated to an
algebraic variety, the natural map $\Pic_\alg (X) \to \rH^2(X,\Z)$
from the algebraic Picard group need not be isomorphism. For example,
consider the case when $X$ is the analytic space associated to a
non-complete and non-singular irreducible complex curve of genus
$g$. By  \cite[Corollary 1.3]{picard}, the analytic Picard group
always vanishes for such curves while the algebraic Picard group
vanishes iff $g=0$.
\end{remark}

\

\section{Analytic models for the cohomological disk algebra $\HPV(X,a)$}
\label{sec:disk}

\noindent Let $(X,W)$ be a Landau-Ginzburg pair of dimension $d$.  
Fix a holomorphic factorization $a=(E,D)$ of $W$ and set
$\bbpd_a:=\bbpd_{a,a}= \bbpd_{End(E)}$, $\md_a:=\md_{a,a}= [D,\cdot]$
and $\updelta_a:=\updelta_{a,a}=\bbpd_a+\md_a$. Recall from Subsection 
\ref{subsec:DiskAlgebra} that we have: 
\be
\Delta_a=\bbpd+\vartheta_a~~,
\ee
where $\bbpd:=\bbpd_{\wedge TX\otimes End(E)}$ and
\be
\vartheta_a\eqdef \ioda_W+\md_a~.
\ee
Notice that $(\vartheta_a)^2=0$ and that $\vartheta_a$ anticommutes
with $\bbpd$. Consider the 2-periodic complex of holomorphic vector bundles: 
\be
(\cP_a): \ldots \stackrel{\vartheta_a}{\longrightarrow} \bigoplus_{s+t=k}\wedge^{|t|} TX\otimes End^{\hat s}(E)\stackrel{\vartheta_a}{\longrightarrow} 
\bigoplus_{s+t=k+1}\wedge^{|t|} TX\otimes End^{\hat s}(E)\stackrel{\vartheta_a}{\longrightarrow} \ldots 
\ee
and the 2-periodic sequence of projective
$\O(X)$-modules:
\ben
\label{H0seq_disk}
(\bP_a):~~\ldots \stackrel{\vartheta_a}{\longrightarrow} \rH^0\big(\,\bigoplus_{s+t=k}\wedge^{|t|} TX\otimes End^{\hat s}(E)\big)\stackrel{\vartheta_a}{\longrightarrow} 
\rH^0\big(\,\bigoplus_{s+t=k+1}\wedge^{|t|} TX\otimes End^{\hat s}(E)\big)\stackrel{\vartheta_a}{\longrightarrow} \ldots ~~,
\een
where $\cP_a^k\eqdef \bigoplus_{s+t=k}\wedge^{|t|} TX\otimes End^{\hat
  s}(E)$ and $\bP_a^k\eqdef \rH^0(\cP_a^k)$ sit in position $k$ and $s\in \Z$, $t\in \Z_{\leq 0}$.

\

\begin{Proposition}
Suppose that $X$ is Stein. Then for each $k$, the $\O(X)$-module
$\HPV^{\hat k}(X,a)$ is naturally isomorphic with the cohomology at
position $k$ of the sequence of projective $\O(X)$-modules \eqref{H0seq_disk}.
\end{Proposition}

\begin{proof}
The bicomplex $(\PV(X,End(E)),\vartheta_a,\bbpd)$ can unwind to a
horizontally 2-periodic bicomplex ${}^{1}\cC^{\bullet,\bullet}$ with
vertical differential $\bbpd$ and horizontal differential
$\vartheta_a$, where:
\be
{}^{1}\cC^{i,j} \eqdef \bigoplus_{t+s=i} \cA^j(X,\wedge^{|t|} 
TX\otimes End^{{\hat s}}(E))~,~~\forall s, j\in \Z~,~~\forall t\in \Z_{\leq 0}~~.
\ee
We have ${}^{1}\cC^{i,j}=0$ unless $j\in \{0,\ldots, d\}$, so this
complex is vertically bounded. Its associated spectral sequence
${}^1\bE\eqdef ({}^{1}\bE_r^{i,j},{}^{1}\dd_r)_{r\geq 0}$ has zeroth
page given by:
\ben
\label{E_3_0}
{}^{1}\bE_0^{i,j} \eqdef {}^{1}\cC^{i,j} ~,~~{}^{1}\dd_0 \eqdef \bbpd~~.
\een
The columns are the Dolbeault resolutions of 
$\cP_a^i=\oplus_{t+s=i}\wedge^{|t|} TX\otimes End^{{\hat s}}(E)$. This spectral
sequence converges since the bicomplex ${}^{1}\cC^{\bullet,\bullet}$ is
vertically bounded. Since $X$ is Stein, Cartan's Theorem B implies
$\rH^j_{\bbpd}(\wedge^{|t|} TX\otimes End^{{\hat s}}(E))=0$ for
$j>0$. Thus on page 1 the spectral sequence is concentrated at the
zeroth row ($j=0$) and has differential ${}^{1}\dd_1\eqdef\vartheta_a$. 
It follows that ${}^{1}d_2=0$, since these differentials are
maps between different rows. Thus ${}^1\bE$ degenerates at
the second page and:
\ben
\rH^k_{\Delta_a} (\PV(X,End(E))) = \rH^k_{\vartheta_a}({}^{1}\bE_1^{\bullet,0}) = \rH^k(\bP_a) \ . 
\een
~\qed
\end{proof}

\

\noindent 
Even without the assumption that $X$ be Stein, the cohomology of the
complex \eqref{H0seq_disk} can itself be computed using another spectral
sequence. Indeed, the decomposition $\vartheta_a = \ioda_W+\md_a$
implies:
\be
\bP_a^{i}=\bigoplus_{s+t=i} {}^{2}\cC^{s,t}~~,
\ee
where ${}^{2}\cC^{\bullet,\bullet}$ is the bicomplex of $\O(X)$-modules defined
through:
\ben
\label{3C_lab}
{}^{2}\cC^{s,t} \eqdef \rH^0(\wedge^{|t|} TX\otimes End^{{\hat s}}(E))~,~~\forall s\in \Z~,
~~\forall t\in \Z_{\leq 0}~~
\een
with vertical differential $\ioda_W$ and horizontal differential
$\md_a$. Consider the following spectral sequence ${}^{2}\bE\eqdef
({}^{2}\bE_r^{s,t},{}^{2}\dd_r)_{r\geq 0}$ with zeroth page defined by:
\ben
\label{2bE0}
{}^{2}\bE_0^{s,t} \eqdef {}^{2}\cC^{s,t}~,~~{}^{2}\dd_0 \eqdef \ioda_W~~,
\een
with ${}^{2}\cC^{\bullet,\bullet}$ as in (\ref{3C_lab}).

\

\begin{Lemma}
The spectral sequence ${}^{2}\bE$ defined above degenerates at
page at most $d+2$ and converges to the cohomology of the
complex of projective $\O(X)$-modules \eqref{H0seq_disk}.
\end{Lemma}

\begin{proof}
The complex ${}^{2}\cC^{\bullet,\bullet}$ is horizontally 2-periodic and
vertically bounded, since ${}^{2}\bE_0^{s,t}$ vanishes for $t\not\in
\{-d,\ldots, 0\}$. The differentials  ${}^2\dd_{r}:{}^{2}\bE_r^{s,t}\rightarrow 
{}^{2}\bE_r^{s+r,t-r+1}$ vanish for $r=d+2$ and all $s$ and $t$  since $t$ 
and $t-r+1$  cannot both lie in $\{-d,\ldots, 0\}$. Hence, the spectral sequence 
is convergent and it degenerates at page at most
$d+2$. By construction, the limit of ${}^{2}\bE$ equals
the cohomology of $\bP_a$.~~\qed
\end{proof}

\smallskip 

\begin{Proposition}
\label{prop:HPVStein}
Suppose that $X$ is Stein. Then the spectral
sequence ${}^{2}\bE$ defined above degenerates at the second page and 
${}^{2}\bE_2$ has nodes given by:
\ben
\label{2bE2}
{}^{2}\bE_2^{s,t}\simeq_{\O(X)}\HPV^t(X,W)\otimes_{\O(X)} \End^{\hat s}_{\HF(X,W)}(a)~~.
\een
\end{Proposition}

\begin{proof}
Recall the sequence $\bK_W$ of projective $\O(X)$-modules defined in
\eqref{H0seq}. Since $X$ is Stein, we have $\rH^0(\wedge^{|t|} TX\otimes End^{{\hat
    s}}(E))\simeq_{\O(X)} \rH^0(\wedge^{|t|} TX)\otimes_{\O(X)} \rH^0(End^{\hat
  s}(E))$ (see \cite[page 403]{ForsterStein}. Hence \eqref{2bE0}
reduces to:
\be
{}^{2}\bE_0^{s,t}=\rH^0(\wedge^{|t|} TX)\otimes_{\O(X)} \rH^0(End^{\hat s}(E))=\bK_W^t\otimes_{\O(X)} \rH^0(End^{\hat s}(E))~~,
\ee
with differential ${}^{2}\dd_0\eqdef\ioda_W$. Thus:
\be
{}^{2}\bE_1^{s,t}=\rH^t(\bK_W)\otimes_{\O(X)} \rH^0(End^{{\hat s}}(E))=\rH^t(\bK_W)\otimes_{\O(X)} \rGamma(X,End^{{\hat s}}(E))
\ee
with first page differentials ${}^{2}\dd_1\eqdef\md_a$. This
implies that the second page has nodes given by: 
\ben
\label{2bE22}
{}^{2}\bE_{2}^{s,t} = \rH^t(\bK_W) \otimes_{\O(X)}\End_{\HF(X,W)}^{\hat s}(a)~~.
\een
Since $X$ is Stein, Theorem \ref{thm:BulkDeg} gives
$\rH^t(\bK_W)\simeq_{\O(X)}\HPV^t(X,W)$, so \eqref{2bE22} reduces to
\eqref{2bE2}.~\qed
\end{proof}

\

\begin{Proposition}
\label{prop:HPVSteinDiscrete}
Suppose that $X$ is Stein and $\dim_\C Z_W=0$. Then the spectral
sequence ${}^{2}\bE$ degenerates at the second page and there exists a
natural isomorphism of $\Z_2$-graded $\O(X)$-modules:
\ben
\label{HPVXa}
\HPV(X,a)\simeq_{\O(X)} \rH_{\md_a}(\rH_{\ioda_W}({}^{2}\cC^{\bullet,\bullet}))\simeq_{\O(X)} \Jac(X,W)\otimes_{\O(X)}\End_{\HF(X,W)}(a)~~.
\een
\end{Proposition}

\begin{proof}
Since $X$ is Stein and $\dim_\C Z_W=0$, Proposition \ref{prop:HJac} gives
$\HPV(X,W)\simeq \Jac(X,W)$, where the right hand side is concentrated
in degree $0$. Now the statement follows easily from Proposition \ref{prop:HPVStein}.
\ 
~\qed
\end{proof}

\

\section{Some examples}
\label{sec:examples}

\noindent In this section, we discuss a few classes of examples which
illustrate the general results obtained above. In particular, we show
that topologically non-trivial holomorphic factorizations are abundant
in B-type topological Landau-Ginzburg models defined on Stein
manifolds and illustrate some of the differences between these and the
algebraic models which are more commonly considered in the literature.

\

\subsection{Domains of holomorphy in $\C^d$}
\label{domain}

Let $X=U\subseteq \C^d$ be a domain of holomorphy\footnote{By the
Cartan-Thullen theorem (see \cite{GR}), a domain $U\subset \C^d$ is
Stein iff it is holomorphically convex, i.e. iff it is a domain of
holomorphy. Notice that domains of holomorphy are very special cases
of Stein manifolds.}. In this case, $U$ is Stein and holomorphically
parallelizable (and hence Calabi-Yau). Moreover, $U$ admits integrable
holomorphic frames given by $u_i=\partial_i$, where $\partial_i:=
\frac{\partial}{\partial z^i}$ and $\{z^1,\ldots, z^d\}$ is a system
of globally-defined complex coordinates on $U$. Assume that $W\in
\O(U)$ has isolated critical points. Then Proposition
\ref{prop:BulkPar} gives:
\be
\HPV(U,W)=\HPV^0(U,W)\simeq_{\O(U)} \Jac(U,W)=\O(U)/\langle \partial_1 W,\ldots, \partial_d W\rangle~~,
\ee
thereby recovering a result of \cite{LLS}. Let us further assume that
$U$ is contractible.  Then any finitely-generated projective
$\O(U)$-module is free\footnote{In this case, any complex vector
bundle defined on $U$ is topologically trivial so the Oka-Grauert
principle (see \cite{OkaGrauert,Oka}) implies that any holomorphic vector
bundle defined on $U$ is holomorphically trivial.}. In this case,
projective analytic factorizations coincide with analytic matrix
factorizations and $\HDF(U,W)$ is equivalent with the $\Z_2$-graded
cohomological category $\HMF(U,W)$ of analytic matrix factorizations
of $W$ by Theorem \ref{thm:BoundaryStein} and Corollary
\ref{cor:MF}. Moreover, Proposition \ref{prop:HPVSteinDiscrete} gives
$\HPV(U,a)\simeq_{\O(U)} \Jac(U,W)\otimes_{\O(U)} \End_{\HF(U,W)}(a)$ for any
holomorphic factorization $a$.  The simplest situation is obtained for
$U=\C^d$. In that case, $\O(U)=\O(\C^d)$ is the ring of entire
functions of $d$ variables, which is already very rich
\cite{Lelong}. For $d=1$, we have $\O(U)=\O(\C)$, which is much larger
than the ring $\O_\alg(\C)=\C[z]$ of univariate polynomials with
complex coefficients. Indeed, Liouville's theorem implies that an
entire function $f\in \O(\C)$ is a polynomial iff it has at most a
pole singularity at infinity. Thus $\O(\C)\backslash \O_\alg(\C)$ 
consists of all entire functions with an essential singularity at
infinity (a simple example of which is the exponential function
$f(z)=e^z$). There appears to exist no good reason (apart from mere
convenience) to require $W$ to be a polynomial.

\

\subsection{Non-compact Riemann surfaces}
\label{RS}

Let $X$ be a smooth, connected and non-compact Riemann surface without
boundary. As recalled in Appendix \ref{app:Stein}, every such surface
is Stein by a result of Behnke and Stein \cite{BS}. Moreover, any
holomorphic vector bundle on $X$ is holomorphically trivial
\cite[Theorem 30.3]{FO}. In particular, $X$ is holomorphically
parallelizable and hence Calabi-Yau. Since $X$ is complex
one-dimensional, holomorphic parallelizability implies existence of a
globally-defined holomorphic vector field $v$ on $X$.  Since any
holomorphic vector bundle defined on $X$ is holomorphically trivial,
it follows that any finitely-generated projective $\O(X)$-module is free,
hence $\HPF(X,W)$ coincides with $\HMF(X,W)$ (See Corollary
\ref{cor:MF}). This also follows from the fact that the ring $\O(X)$
is a B\'ezout domain (see \cite{Helmer2,Henriksen3,Alling1,Alling2})
and from the fact that any finitely-generated projective module over a
B\'ezout domain is free (see \cite{FS}). Matrix factorizations over
B\'ezout domains were studied in \cite{bezout}.

\

\begin{Proposition}
\label{Riemann}
Let $W\in \O(X)$ be a non-constant holomorphic function. Then: 
\be
\HPV(X,W)=\HPV^0(X,W)\simeq_{\O(X)} \Jac(X,W)=\O(X)/\langle v(W) \rangle~~,
\ee
where $v$ is any non-trivial globally-defined holomorphic vector field on $X$.
Moreover, there exist equivalences of $\O(X)$-linear $\Z_2$-graded
categories:
\be
\HDF(X,W)\simeq \HF(X,W)\simeq \HF_\free(X,W)\simeq \HMF(X,W)~~.
\ee
For any holomorphic factorization $a$ of $W$, we have: 
\be
\HPV(X,a)\simeq_{\O(X)} \Jac(X,W)\otimes_{\O(X)}  \End_{\HF(X,W)}(a)~~.
\ee
\end{Proposition}

\begin{proof}
Notice that $\dim_\C Z_W=0$ since $W$ is non-constant. The conclusion
now follows immediately from Proposition \ref{prop:BulkPar}, Theorem
\ref{thm:BoundaryStein}, Corollary \ref{cor:MF} and Proposition
\ref{prop:HPVSteinDiscrete} upon using the observations made
above.~\qed
\end{proof}

\

\noindent Notice that the non-compact Riemann surface $X$ need not be
(affine) algebraic; in particular, it can have infinite genus and an
infinite number of ends. 

\

\subsection{Application to D-branes in topological deformations 
of B-twisted $A_n$ minimal models}

An {\em algebraic-geometric} version of the category of topological
D-branes in B-twisted $A_n$ minimal models and their deformations was
considered in the physics literature (see, for example,
\cite{BHLS,KL2,HLL}). In such models, $X$ is the complex plane $\C$
and $W$ is a univariate polynomial function.  As pointed out in
\cite{LG2}, the framework used in \cite{HLL} (and later on in the
mathematics literature on Landau-Ginzburg models) can be generalized
in models which are away from the conformal point since in a
Landau-Ginzburg model which is not required to be scale-invariant
there is no reason to restrict $W$ to be a polynomial or to restrict
the ring $\O(\C)$ of entire functions to the ring $\O_\alg(\C)\simeq
\C[z]$ of regular algebraic functions defined on $\C$, where $\C[z]$
is the ring of univariate polynomials. In these examples (and more
generally when $X$ is a smooth non-compact Riemann surface and $W$ is
a holomorphic function which has only a finite number of multiple
zeros) one can show, however, that replacing $\O(X)$ with $\O_\alg(X)$
leads to essentially the same result for the homotopy category of
topological D-branes.

Restricting to the case $X=\C$, let us assume that the entire function
$W\in \O(\C)$ is ``critically finite'' in the sense that it has only a
finite number of multiple zeroes (though it can have an infinite
number of simple zeroes). Since we are in the Stein case, this
assumption is justified by the requirement (see \cite{nlg1}) that the
on-shell bulk and boundary state spaces of the topological
Landau-Ginzburg model be finite-dimensional (see Remark
\ref{rem:SteinFinite}). By the discussion of the previous subsection,
we have $\HF(\C,W)\simeq \HMF(\C,W)$. Consider the free
$\O(\C)$-supermodule $M=M^\0 \oplus M^\1$, where $M^\0=M^\1=\O(\C)$.
For each zero of $W$ of multiplicity $k>1$, there exist $k-1$
``primary holomorphic matrix factorizations'' of $W$ having $M$ as
their underlying supermodule.  Using more general results about
elementary divisor domains, one can show \cite{edd} that these special
factorizations additively generate the homotopy category
$\mathrm{hmf}(\C,W)=\HMF^\0(\C,W)$ of finite rank holomorphic matrix
factorizations over $\C$, which in turn as a triangulated category is
generated by a single object; this follows from more general results
which we establish in upcoming work for any elementary divisor
domain. These results show that the holomorphic and algebraic homotopy
categories of matrix factorizations of ``critically finite''
holomorphic superpotentials $W$ can be identified as triangulated
categories, even though the ring $\O(\C)$ is not Noetherian.

\

\subsection{Complements of affine hyperplane arrangements}
\label{subsec:ha}

The theory of affine hyperplane arrangements (see
\cite{Dimca,Yuzvinsky}) allows one to construct a large class of
parallelizable Stein manifolds of non-trivial topology and with a
strong combinatorial flavor\footnote{The Stein property of these and
similar examples follows from the general fact \cite{DG} that the
complement of an analytic hypersurface in a Stein manifold is
Stein. For complements of complex affine hyperplane arrangements, it
also follows from the fact that they are analytic spaces associated to
non-singular complex affine varieties.}.  Let $\cA$ be a
$d$-dimensional central complex affine hyperplane arrangement, i.e. a
finite collection of distinct hyperplanes $H$ in $\C^d$, each of which
passes through the origin. Let $\alpha_H:\C^d\rightarrow \C$ be
defining linear functionals for the hyperplanes of the
arrangement. Let $X$ denote the complement of the set $\cup_{H\in
\cA}H$ in $\C^d$. Then $X$ is a parallelizable (and hence Calabi-Yau)
Stein manifold. Moreover, $X$ is the analytic space associated to a
non-singular complex affine variety which can be realized as the
hypersurface in $\C^{d+1}$ given by the equation:
\be
x_{d+1}\prod_{H\in \cA}{\alpha_H(x_1,\ldots, x_d)}=1~~.
\ee
The cohomology ring $\rH(X,\Z)$ is isomorphic with the {\em
Orlik-Solomon algebra} of $\cA$, which is defined as follows. Fix a
total ordering of $\cA$ and let $E$ be the exterior $\Z$-algebra on
the degree one generators $(e_H)_{H\in \cA}$ and $\pd$ be the degree
$-1$ differential on $E$ which is determined uniquely by the condition
$\pd(e_H)=1$ for all $H\in \cA$. For any non-empty subset $S$ of
$\cA$, let $e_S$ be the corresponding Grassmann monomial with respect
to the total order chosen on $\cA$. We say that $S$ is {\em dependent}
if $\cap_{H\in S}H\neq \emptyset$ and the defining linear functionals
$\alpha_H$ of the hyperplanes $H\in S$ are linearly dependent. Let $I$
be the homogeneous ideal of $E$ generated by the elements
$\partial(e_S)$ with dependent $S$. The Orlik-Solomon algebra of $\cA$
is the $\Z$-graded quotient algebra $A \eqdef E/I$. One can show that
each homogeneous component $A^k$ of $A$ is a free $\Z$-module of
finite rank (see \cite[Corollary 3.1]{Dimca}).  A classical result due
to \cite{Brieskorn,OS} (see \cite[Theorem 3.5]{Dimca}) states that the
cohomology ring $\rH(X,\Z)$ is isomorphic with $A$ as a graded
$\Z$-algebra; in particular, each cohomology group $\rH^k(X,\Z)$ is a
free $\Z$-module of finite rank. The {\em intersection poset $L$} of
$\cA$ is the set of all those subspaces $F$ of $\C^d$ (called {\em
flats}) which arise as finite intersections of hyperplanes from $\cA$,
ordered by reverse inclusion. This poset contains $\C^d$ (the
intersection of the empty set of hyperplanes) as its least
element. Since we assume that $\cA$ is central, $L$ is in fact a
bounded lattice with greatest element given by $\cap_{H\in \cA}H$;
moreover, it is a geometric lattice (see \cite[Theorem 2.1]{Dimca}).

Let $P_X(t)\eqdef \sum_{j=0}^d \rk \rH^j(X,\Z) \, t^j$
denote\footnote{Notice that $\rk \rH^j(X,\Z)=\rk \rH_j(X,\Z)$ by the
universal coefficient theorem, since both $\rH^j(X,\Z)$ and
$\rH_j(X,\Z)$ are finitely-generated.} the Poincar\'e polynomial of
$X$. The following result (see \cite[Theorem 2.2 and Corollary
3.6]{Dimca}) or \cite[Corollary 3.4]{Yuzvinsky}) allows one to
determine the ranks of the cohomology groups $\rH^j(X,\Z)$:

\

\begin{Proposition}
Let $\mu_L$ be the M\"{o}bius function of the locally-finite poset
$L$. Then $\mu_L(\C^d,F)\neq 0$ for all flats $F\in L$ and the sign of
$\mu_L(\C^d,F)$ equals $(-1)^{\codim_\C F}$. Moreover, we have:
\be
P_X(t)=\sum_{F \in L} |\mu_L(\C^d,F)| t^{\codim_\C F}=\sum_{F \in L} \mu_L(\C^d,F) (-t)^{\codim_\C F}~~.
\ee
\end{Proposition}

\

\noindent In particular, one has $\rk \rH^2(X,\Z)=\sum_{F\in
L:\codim_\C F=2}|\mu_L(\C^d,F)|$, so $\rH^2(X,\Z)$ will be non-zero
when $L$ contains codimension two flats; since $\cA$ is central, this
happens iff $\cA$ contains two distinct hyperplanes. In fact, one has
$\rk \rH^j(X,\Z)>0$ for all $j=0,\ldots, \rk \cA$ and $\rk
\rH^j(X,\Z)=0$ for $j>\rk \cA$ (see \cite[Chap. 2.5, Exercise
2.5]{Dimca}), where $\rk\cA\eqdef \codim_\C \left[\cap_{H\in \cA}
H\right]$ is the {\em rank} of $\cA$; hence $\rk\cA\geq 2$ implies
$\rH^2(X,\Z)\neq 0$. By the construction of Subsection
\ref{subsec:nontriv}, it follows that the complement of any $d$-dimensional central
complex affine hyperplane arrangement $\cA$ with $d\geq 2$ and $\rk\cA\geq 2$ admits
topologically non-trivial elementary holomorphic factorizations.

\

\begin{example}
\label{ex:ha} Let $\cA$ be the hyperplane arrangement defined in
$\C^3$ by the six linear functionals $x,y,z,x-y,x-z$ and $y-z$ and let
$X=\C^3\setminus \left(\cup_{H\in \cA}H\right)$. Then $P_X(t)=1+6 t+11 t^2+6 t^3$ (see
\cite[Example 3.5]{Yuzvinsky}). In particular, we have $\rH^2(X,\Z)=\Z^{11}$.
\end{example}

\

\begin{example}
\label{ex:ba} 
Let $\cA$ be the {\em Boolean arrangement} defined in $\C^d$ (where
$d\geq 2$) by the linear functionals $x_1,\ldots, x_d$, i.e. by the
equation $x_1\cdot\ldots \cdot x_d=0$. Thus $\cA$ is the union of the
coordinate hyperplanes $x_i=0$ and the parallelizable Stein manifold
$X=\C^d\setminus \left(\cup_{H\in \cA}H\right) =(\C^\ast)^d$ is a
complex algebraic torus. The ideal $I$ vanishes and the Orlik-Solomon
algebra coincides with the Grassmann $\Z$-algebra $E$ on $d$
generators. The Poincar\'e polynomial of $X$ is given by
$P_X(t)=(t+1)^d$ (see \cite[Example 2.12 and Example 3.3]{Dimca}),
thus $\rH^2(X,\Z)\simeq \Z^{\frac{d(d-1)}{2}}$. Writing $x_k=r_k
e^{2\pi \i\theta_k}$ with $r_k>0$ and real $\theta_k$ shows that $X$
is homotopy-equivalent with the $d$-dimensional torus $T^d=\{x\in
(\C^\ast)^d \,\big|\, |x_1|=\ldots =|x_d|=1\}\simeq (\rS^1)^d$
parameterized by $\theta_1,\ldots,\theta_d$. The map $\pi:X\rightarrow
T^d$ given by:
\be
\pi(x_1,\ldots, x_d)=\Big(\frac{x_1}{|x_1|},\ldots, \frac{x_d}{|x_d|}\Big)
\ee
gives a homotopy retraction of $X$ onto $T^d$, so pull-back by $\pi$
induces an isomorphism from $\rH^2(T^d,\Z)$ to $\rH^2(X,\Z)$ and hence
(by the Oka-Grauert principle) also an isomorphism from the group
$\Prin_{\rU(1)}(T^d)$ of principal $\rU(1)$ bundles on $T^d$ to the
group $\Pic_\an(X)$ of isomorphism classes of holomorphic line bundles
on the Stein manifold $X$. Notice that $X$ can be embedded in
$\C^{d+1}$ as the affine hypersurface given by the equation $x_1\cdot \ldots
\cdot x_{d+1}=1$. Thus $X$ is the analytic space associated to the
complex affine variety:
\be
X_\alg\eqdef \Spec (\C[x_1,\ldots, x_d,x_1^{-1},\ldots, x_d^{-1}]=
\Spec\Big(\C[x_1,\ldots, x_{d+1}]/\langle x_1\ldots x_{d+1}-1\rangle \Big)~.
\ee
In this case, the algebraic Picard group $\Pic_\alg(X)$
vanishes\footnote{This follows from the fact that the Laurent
polynomial ring $\C[x_1,\ldots, x_d,x_1^{-1},\ldots x_d^{-1}]$ is a
unique factorization domain (UFD) (being the localization of the Noetherian 
UFD $\C[x_1,\ldots,
x_d]$ at a multiplicative set which doesn't contain zero) and from the
fact that the Picard group of any UFD vanishes. It also follows from a
much more general result due to Weibel (see \cite[Theorem
1.6]{Weibel}).} while $\Pic_\an(X)\simeq \Z^{\frac{d(d-1)}{2}}$. When
$d\geq 2$, the construction of Subsection \ref{subsec:nontriv} shows
that $X$ supports topologically non-trivial elementary holomorphic
factorizations, whose associated analytic factorizations are therefore
described by non-free finitely-generated projective
$\O(X)$-modules. Since its algebraic Picard group vanishes, $X$ admits
only topologically trivial algebraic elementary factorizations. This
illustrates the difference between holomorphic and algebraic B-type LG
models.
\end{example}

\

\subsection{Analytic complete intersections}
\label{subsec:ci}

Let $X\subset \C^N$ be an analytic complete intersection of complex
dimension $d$, defined by the regular sequence $f_1, \dots ,
f_{N-d}\in \O(\C^N)$. Then $X$ is holomorphically parallelizable by
results of \cite{Fo} (see Theorem \ref{thm:CI}). In particular, $X$ is
Calabi-Yau. Let $u_1,\ldots,u_d$ be a globally-defined holomorphic frame of $TX$. 

\

\begin{Proposition}
\label{ci}
Let $W \in \O(X)$ be a holomorphic function such that $\dim_\C Z_W=0$. Then:
\beqa
\HPV(X,W)=\HPV^0(X,W)\simeq_{\O(X)} \Jac(X,W)=\O(X)/\langle u_1(W),\ldots, u_d(W)\rangle ~~.
\eeqa
Moreover, we have equivalences of $\O(X)$-linear  $\Z_2$-graded
categories:
\be
\HDF(X,W)\simeq \HF(X,W)\simeq \HPF(X,W)~~.
\ee
For any holomorphic factorization $a$ of $W$, we have: 
\be
\HPV(X,a)\simeq_{\O(X)}  \Jac(X,W)\otimes_{\O(X)}  \End_{\HF(X,W)}(a) ~~.
\ee
\end{Proposition}

\begin{proof}
Follows immediately from Proposition \ref{prop:BulkPar},
Theorem \ref{thm:BoundaryStein} and Proposition \ref{prop:proj} 
using the observations made above. ~\qed
\end{proof}

\

\noindent We have $\O(X)\simeq \O(\C^N)/\langle f_1,\ldots,
f_{N-d}\rangle$. In particular, any superpotential $W\in \O(X)$ is the
restriction of a holomorphic function $\cW\in \O(\C^N)$ defined on the
ambient affine space. However, the critical locus $Z_W$ of $W=\cW|_X$
on $X$ need not\footnote{As a simple example, consider the linear
subspace $X$ defined by the single equation $x=0$ in $\C^2$ and the
holomorphic function $\cW(x,y)=xy$. Then $\cW$ has the origin as its
single critical point, but its restriction $W$ to $X$ vanishes, so
we have $Z_W=X$.}  coincide with the intersection with $X$ of the critical
locus $Z_\cW$ of $\cW$. Since $Z_\cW$ is defined by the
equation $\pd \cW=0$ on $\C^N$, while $Z_W$ is defined by the equation
$\pd W=\pd \cW|_{TX}=0$, we always have $Z_\cW\cap X\subseteq Z_W$,
but the inclusion can be strict. The critical locus of $W$ can be
determined as follows. The holomorphic tangent space to $X$ at a point
$x\in X$ equals $\cap_{i=1}^{N-d} \ker(\pd_x f_i)$. Thus $T_xX$
identifies with the space of solutions of the linear system:
\ben
\label{tsys}
(\pd_x f_1)(v)= \ldots =(\pd_x f_{N-d})(v)=0~~(v\in T_x \C^N)~~.
\een
Identifying $T_x\C^N$ with $\C^N$, we write $v=(v^1,\ldots, v^N)$.
Then \eqref{tsys} becomes a system of $N-d$ linear equations
for $N$ unknowns: 
\ben
\label{Xsys} 
\sum_{j=1}^N (\partial_j f_i)(x)v^j=0~~(i=1,\ldots, N-d)~~.  
\een 
Since $X$ is smooth, the matrix of this linear system has maximal
rank. Since $X$ is parallelizable, we have $TX\simeq \cO_X^{\oplus
d}$, so there exist globally-defined holomorphic functions $u_1,\ldots,
u_d:X\rightarrow \C^N$ such that $u_k(x)=(u_k^1(x),\ldots, u_k^N(x))$
(with $k=1,\ldots, d$) is a basis of solutions of \eqref{Xsys} for any
$x\in X$. We have:
\be
\pd W= \pd\cW|_{T X}=\sum_{k=1}^d \left(\sum_{j=1}^N
u^j_k \partial_j \cW \right) u^k~~, 
\ee
where $u^1,\ldots, u^d$ is the dual coframe of $T^\ast X$. Then
$Z_W$ is defined by the equations:
\ben
\label{Zci}
f_1=\ldots =f_{N-d}=\sum_{j=1}^N u^j_1 \partial_j \cW=\ldots=\sum_{j=1}^N u^j_d \partial_j \cW=0~~.
\een
For a generic choice of $\cW$, this system of $N$ equations in $\C^N$
has a zero-dimensional set of solutions, which we assume to be the
case from now on (technically, we are assuming that the $N$
holomorphic functions appearing in \eqref{Zci} form a regular
sequence, i.e. that $Z_W$ is a zero-dimensional complete intersection
in $\C^N$). Then Proposition \ref{ci} shows that $\HPV(X,W)$ is
concentrated in degree zero and isomorphic with the Jacobi algebra:
\be
\Jac(X,W)=\O(X)/\langle u_1(W),\ldots, u_d(W)\rangle\simeq \O(\C^N)
\Big/\big\langle f_1,\ldots, f_{N-d}\,,\,\sum_{j=1}^N u_1^j \pd_j \cW\,,\, \ldots\,, \,\sum_{j=1}^N u_d^j \pd_j \cW\big\rangle~~.
\ee
A globally-defined holomorphic frame of $TX$ can be difficult to
compute in practice, so it is often more convenient to find $r$
globally-defined holomorphic vector fields on $X$ (where $r\geq d$)
which generate the fibers of $TX$. These can be viewed as holomorphic
functions $v_\ell:X\rightarrow \C^N$ (where $\ell=1,\ldots, r$) such
that $v_\ell(x)=(v_\ell^1(x),\ldots, v_\ell^N(x))$ generate the space
of solutions of \eqref{Xsys} for all $x\in X$. Then the critical
ideal $\J(X,W)=\langle u_1(W),\ldots, u_d(W)\rangle \subset \O(X)$ of
$W$ (see Lemma \ref{lemma:par}) is generated by the elements $\pd_{v_\ell}\cW|_X=\sum_{j=1}^N
v^j_\ell \partial_j \cW\in \O(X)$ and the Jacobi algebra can be written
as:
\ben
\label{Jacv}
\Jac(X,W)\simeq \O(\C^N)\Big/\big\langle f_1,\ldots, f_{N-d}\,,\,\sum_{j=1}^N v_1^j \pd_j \cW\,,\, \ldots\,, \,\sum_{j=1}^N v_r^j \pd_j \cW\big\rangle~~,
\een
while $Z_W$ is defined by the equations: 
\ben
\label{Zcigen}
f_1=\ldots =f_{N-d}=\sum_{j=1}^N v^j_1 \partial_j \cW=\ldots=\sum_{j=1}^N v^j_r \partial_j \cW=0~~.
\een
A particularly simple presentation of the Jacobi algebra exists for
smooth hypersurfaces in $\C^N$, of which we will give an explicit
example below. In this case, we have:

\

\begin{Proposition}
\label{prop:3hyp}
Let $X$ be a non-singular hypersurface in $\C^N$ defined by the
equation $f=0$. Then the holomorphic tangent vector fields
$(v_{ij})_{1\leq i<j\leq N}$ defined on $X$ through:
\be
v_{ij}^{k}=(\partial_j f)\delta_{ik} -(\partial_i f) \delta_{jk}~~(k=1,\ldots, N)
\ee
(where $\delta$ is the Kronecker symbol) generate each fiber of
$TX$. Moreover, if $\cW\in \O(\C^N)$ is a holomorphic function, then
the critical locus $Z_W$ of the restriction $W\eqdef \cW|_X$ is
defined by the following system of equations in $\C^N$:
\beqan
\label{ZWSurface}
& f=0\nn\\
& \partial_i \cW\partial_j f -\partial_j \cW\partial_i f=0~~(1\leq i<j\leq N)~~.
\eeqan
If $W$ has isolated critical points on $X$, then we have:
\ben
\label{JacSurface}
\Jac(X,W)=\O(\C^N)/I~~,
\een
where the ideal $I$ of $\O(\C^N)$ is generated by $f$ and by the
holomorphic functions $\partial_i \cW\partial_j f -\partial_j
\cW\partial_i f$ with $1\leq i<j\leq N$.
\end{Proposition}

\begin{proof}
It is clear that $v_{ij}$ are orthogonal to $(\partial_1f,
\ldots, \partial_Nf)$ and hence tangent to $X$. Let $A$ be the
$N(N-1)\times N$ matrix having $(v_{ij})_{1\leq i<j\leq N}$ and as
$(-v_{ij})_{1\leq i<j\leq N}$ its rows.  For each $1\leq i\leq N$,
this matrix has an $(N-1)\times (N-1)$ minor of the form\footnote{We
give the form of this minor for the case $2<i<N$; the cases $i=1$ and
$i=N$ are similar.}:
\be
\Delta_{i}=\left[\begin{array}{ccccccccc} 
\pd_i f & \dots & \dots & \dots  & 0 &\dots &\dots &0\\
0 & \pd_i f &\dots  & \dots & 0 &\dots & \dots & 0\\
\dots & \dots  & \dots &  \dots & \dots & \dots & \dots & \dots\\
0 & \dots  & \dots & \pd_i f & 0 &\dots &\dots &0\\

0 &\dots &\dots & \dots &  -\pd_i f & \dots & \dots & 0\\
0 &\dots &\dots & \dots & 0 & -\pd_i f & \dots & 0\\
\dots &\dots &\dots & \dots & \dots & \dots  & \dots & \dots\\
0 &\dots &\dots & \dots &0 & \dots & \dots & -\pd_i f
\end{array}\right]~~,
\ee
where the square block containing $\pd_if$ on the diagonal has size $i-1$
while the block containing $-\pd_if$ on the diagonal has size $N-i$.
The determinant of this minor is $\det
\Delta_{i}=(-1)^{N-i}(\pd_if)^{N-1}$.  Since $f$ is non-singular,
these minors cannot vanish simultaneously on $X$. Hence $A$ has rank
$N-1$ on $X$, thus $(v_{ij}(x))_{1\leq i<j\leq N}$ generate $T_xX$ for
all $x\in X$. Applying \eqref{Jacv} and \eqref{Zcigen} gives
\eqref{ZWSurface} and \eqref{JacSurface}.\,\qed
\end{proof}

Any holomorphic factorization $(E,D)$ of $\cW$ on the
ambient space $\C^N$ restricts to a holomorphic factorization
$(E|_X,D|_X)$ of $W$ on $X$. Since $\C^N$ is Stein and contractible,
the $\Z_2$-graded holomorphic vector bundle $E$ is necessarily
trivial, so this construction produces only topologically trivial
holomorphic factorizations of $W$ on $X$. In general, restrictions of
holomorphic factorizations defined on the ambient space do not recover
all holomorphic factorizations of $W$, since an analytic complete
intersection can support topologically non-trivial holomorphic
factorizations. Indeed, Theorem \ref{thm:CI} states that a Stein
manifold $X$ is parallelizable iff it admits an embedding as an
analytic complete intersection in some $\C^N$. In particular, the
complement $X$ of any central complex affine hyperplane arrangement
$\cA$ (see Subsection \ref{subsec:ha}) can be realized as an analytic
complete intersection. When $\rk \cA\geq 2$, the discussion of
Subsection \ref{subsec:ha} assures us that $\rH^2(X,\Z)\neq 0$, which
by the construction of Subsection \ref{subsec:nontriv} shows that
there exist analytic complete intersections which admit topologically
non-trivial holomorphic factorizations. An explicit holomorphic
factorization of this type is given in Example \ref{ex:ci} below.

\

\begin{example}
Consider the non-singular hypersurface $X$ defined in $\C^3$ by the
equation $f(x_1,x_2,x_3)=x_1e^{x_2}+x_2e^{x_3}+x_3e^{x_1}=0$. Let
$\cW\in \O(\C^3)$ be the holomorphic function given by
$\cW(x_1,x_2,x_3)=x_1^{n+1}+x_2x_3$ (where $n\geq 1$) and $W\in \O(X)$
be the restriction of $\cW$ to $X$. The critical locus of $\cW$
coincides with the origin of $\C^3$, which lies on $X$. As a
consequence, $Z_W$ contains the point $(0,0,0)\in X$. In this example,
we have:
\be
\pd_1 f=e^{x_2}+x_3 e^{x_1}~,~\pd_2 f=e^{x_3}+x_1 e^{x_2}~,~\pd_3 f=e^{x_1}+x_2 e^{x_3}~~.
\ee
The vector fields of Proposition \ref{prop:3hyp} are given by:
\beqan
\label{vk}
&&v_{23}=(0, \pd_3 f, -\pd_2 f)=\big(0, e^{x_1}+x_2 e^{x_3}, -e^{x_3}-x_1 e^{x_2} \big)~~\nn\\
&&v_{13}=(\pd_3 f, 0, -\pd_1 f)=\big(e^{x_1}+x_2 e^{x_3}, 0, -e^{x_2}-x_3 e^{x_1}\big)~~\\
&&v_{12}=(\pd_2 f, -\pd_1 f ,0)=\big(e^{x_3}+x_1 e^{x_2}, -e^{x_2}-x_3 e^{x_1},0 \big)~~\nn
\eeqan
and the defining equations \eqref{ZWSurface} of $Z_W$ take the form:
\beqan
\label{Zeq1}
& x_1 e^{x_2}+x_2 e^{x_3} +x_3 e^{x_1} =0~~\nn\\
& (n+1)x_1^n(e^{x_1}+x_2 e^{x_3})-x_2 (e^{x_2}+ x_3 e^{x_1})=0~~\nn\\
& x_3 (e^{x_1}+x_2 e^{x_3}) - x_2 (e^{x_3}+x_1 e^{x_2})=0~~\\
& (n+1) x_1^n (e^{x_3} +  x_1e^{x_2}) - x_3 (e^{x_2} + x_3 e^{x_1} )=0~~.\nn
\eeqan
The $\O(X)$-algebra $\HPV(X,W)$ is concentrated in degree zero and isomorphic
with the Jacobi algebra $\Jac(X,W)=\O(\C^3)/I$, where $I\subset \O(\C^3)$
is the ideal generated by the four holomorphic functions appearing in the
left hand side of \eqref{Zeq1}. Numerical study
shows that, for generic $n\geq 1$, the transcendental system
\eqref{Zeq1} admits solutions different from $x_1=x_2=x_3=0$, so the
restricted superpotential $W$ generally has critical points on $X$
which differ from the origin. Table \ref{table:critpoints} shows a few
such points for small values of $n$.

\begin{table}[H]
\centering
\begin{tabular}{|c|c|}
\toprule
$n$ & $x$  \\
\midrule\midrule
 1  & $~(0.512, -0.505,1.957)~,~(2.048, -2.114, 2.017)$ ~, \\
  & $~(0.450+0.985 \,\i, -0.241-0.613 \,\i, -0.848+0.747 \,\i )$~,~ $\ldots$  \\
 \hline
 2 & $( -0.435,  -0.109, 2.314)~,~( -0.385,  0.315,  0.207)~,~
(0.604,  -0.553,  1.960)$~,   \\
& $(-0.338-0.599\,\i, ~0.056+0.370\,\i,~ 0.050+0.678\,\i )$~,  $\ldots$  \\
\hline
 3 & $(0.658,  -0.583, 1.963)$~,~ \\
 & $(0.112-0.298\,\i, -0.075+0.122\,\i, -0.089+0.121\,\i )$~,  $\ldots$ \\
 \bottomrule
 \end{tabular}
 \caption{Some non-zero critical points of $W$ on $X$
 for low values of $n$ (4 significant digits).}
\label{table:critpoints}
  \end{table}

\noindent As explained above, any holomorphic factorization of $\cW$ on the
ambient space $\C^3$ induces by restriction a topologically trivial
holomorphic factorization of $W$ on $X$. For any $k\in\{0,\ldots,
n+1\}$, an example of holomorphic (in fact, algebraic) factorization
$(E,D_k)$ of $\cW$ on $\C^3$ is given by (see \cite{quiver})
$E^\0=E^\1=\cO_{\C^3}^{\oplus 2}$ with:
\be
D_k=\left[\begin{array}{cc} 
0 & b_k \\
a_k & 0
\end{array}\right]~~,~~\mathrm{where}~~ a_k=\left[\begin{array}{cc} 
x_2 &~ x_1^{n+1-k} \\
x_1^k & -x_3
\end{array}\right]~~\mathrm{and}~~b_k=\left[\begin{array}{cc} 
x_3 &~ x_1^{n+1-k} \\
x_1^k & -x_2
\end{array}\right]~~.
\ee
This induces a holomorphic factorization $(E|_X,D_k|_X)$ of
$W$ upon restriction to $X$.

\begin{remark}
For any $k\in \{1,2,3\}$, let $X_k$ be the open submanifold of $X$ defined by:
\be
X_k\eqdef \{x\in X \,|\, (\pd_k f)(x)\neq 0\}~~
\ee
and consider the vector fields $e_1^{(k)}$ and $e_2^{(k)}$ defined on $X_k$ through:
\beqa
&& e_1^{(1)}=\frac{v_{13}}{\pd_1 f}=(\frac{\pd_3f}{\pd_1f},0,-1)~~,~~
e_2^{(1)}=\frac{v_{12}}{\pd_1 f}=(\frac{\pd_2f}{\pd_1f},-1,0)~~,\nn\\
&& e_1^{(2)}=\frac{v_{23}}{\pd_2 f}=(0,\frac{\pd_3f}{\pd_2 f},-1)~~,~~
e_2^{(2)}=\frac{v_{12}}{\pd_2 f}=(1, -\frac{\pd_1 f}{\pd_2 f},0)~~,\nn\\
&&e_1^{(3)}=\frac{v_{13}}{\pd_3 f}=(1,0,-\frac{\pd_1 f}{\pd_3 f})~~,~~
e_2^{(3)}=\frac{v_{23}}{\pd_3 f}=(0,1,-\frac{\pd_2 f}{\pd_3 f})~~.\nn
\eeqa
Then $e_1^{(k)},e_2^{(k)}$ form a globally-defined holomorphic frame
of $TX_k$. Since each $X_k$ is parallelizable and Stein (being the
complement of an analytic hypersurface inside the parallelizable Stein
manifold $X$), the three pairs $(X_k,W_k)$ (where $W_k\eqdef
W|_{X_k}$) define restricted B-type Landau-Ginzburg models. Since $X$
is non-singular, the partial derivatives $\pd_k f$ cannot vanish
simultaneously on $X$, which implies that $X_1$, $X_2$ and $X_3$ give
an open cover of $X$. One can think of the model corresponding to
$(X,W)$ as being ``glued'' from these three ``smaller'' models. 
\end{remark}

\end{example}

\

\begin{example}
\label{ex:ci} 
Let $X\simeq (\C^\ast)^2$ be the complement of the $2$-dimensional
Boolean hyperplane arrangement (see Example \ref{ex:ba} of Subsection
\ref{subsec:ha}). As explained there, $X$ can be embedded in $\C^3$
as the analytic hypersurface given by the equation $x_1 x_2 x_3=1$ and
we have $\Pic_\alg(X)=0$ but $\Pic_\an(X)\simeq \rH^2(X,\Z)\simeq
\Z$. One of the two generators of $\rH^2(X,\Z)$ is realized by a
non-algebraic holomorphic vector bundle $L$ which has the property
that the circle bundle defined by any Hermitian metric on the
restriction of $L$ to the real 2-torus $T^2=\{(x_1,x_2)\in (\C^\ast)^2
\, \big| \, |x_1|=|x_2|=1\}\subset X$ generates the group
$\Prin_{\rU(1)}(T^2)\simeq \rH^2(T^2,\Z)$ of isomorphism classes of
principal $\rU(1)$-bundles on $T^2$. Then the opposite generator of
$\rH^2(X,\Z)$ is realized by the dual line bundle $L^{-1}$. The
construction of Subsection \ref{subsec:nontriv} shows that $X$
supports projective analytic factorizations which are not free.

To describe $L$ explicitly, notice that $X$ can be written as the
quotient $\C^2/\Z^2$, where $\Z^2$ acts on $\C^2$ by:
\be
(n_1,n_2)\cdot (z_1,z_2)\eqdef (z_1+n_1,z_2+n_2)~~,
~~\forall (z_1,z_2)\in \C^2~,~\forall (n_1,n_2)\in \Z^2~~.
\ee
The canonical surjection $\pi:\C^2\rightarrow X=(\C^\ast)^2$ is given explicitly by:
\be
x_1=e^{2\pi \i z_1}~~,~~x_2=e^{2\pi \i z_2}~~.
\ee
The restriction of $\pi$ to the two-dimensional real subspace
$V_0\subset \C^2$ defined by the equations $\Im z_1=\Im z_2=0$ induces
a bijection between $V_0/\Z^2$ and $T^2=\{(e^{2\pi \i
\theta_1},e^{2\pi \i \theta_2}) \,|\,\theta_1,\theta_2\in \R\}$; under this
bijection, $\theta_1$ and $\theta_2$ can be viewed as real coordinates
on $V_0$. Consider the surjective maps $\varphi^\pm:\C^2\rightarrow
\C$ given by:
\be
\varphi^\pm(z_1,z_2)=z_1\pm \i z_2~~,
\ee
which restrict to different bijections $\varphi^\pm_0$
between $V_0$ and $\C$:
\be
\varphi^\pm_0(\theta_1,\theta_2)=\theta_1\pm \i \theta_2~~.
\ee  
Consider the lattice $\Lambda:=\Z\oplus \i \Z\subset \C$. Then
$\varphi^\pm$ descend to well-defined surjections
$\bvarphi^\pm:X\rightarrow X_0$, where $X_0:=\C/\Lambda$ is an
elliptic curve of modulus $\tau=\i$. In turn, these maps restrict to
diffeomorphisms $\bvarphi^\pm_0:T^2\rightarrow X_0$ which allow us to
pull-back the complex structure of $X_0$ to two mutually opposite
complex structures on $T^2$ and hence to view the latter in two
different ways as an elliptic curve of modulus $\tau=\i$ which is
biholomorphic with $X_0$. Explicitly, the complex coordinate $z$ on the
covering space $\C$ of $X_0$ is given by $z=\varphi^\pm_0(\theta_1,
\theta_2)=\theta_1\pm \i \theta_2$ in terms of the real coordinates
$\theta_1$ and $\theta_2$ of the covering space $V_0$ of $T^2$. It is
easy to see that $\bvarphi^+$ and $\bvarphi^-$ are homotopy retractions
of $X$ onto $X_0$. Hence $\bvarphi^\pm$ induces two isomorphisms from
$\rH^2(X_0,\Z)$ to $\rH^2(X,\Z)$ which differ by sign. Since
$\bvarphi^\pm$ are holomorphic and $\rH^2(X,\Z)\simeq \Z$ classifies
holomorphic line bundles on $X$, it follows that the $\bvarphi^\pm$
pullback of any topologically nontrivial holomorphic line bundle
defined on $X_0$ is a topologically non-trivial holomorphic line
bundle defined on $X$. More precisely, pullback by $\bvarphi^\pm$
induces an isomorphism:
\be
\bvarphi^\pm:\NS(X_0)\rightarrow \rH^2(X,\Z)~~,
\ee
where $\NS(X_0)\eqdef
\Pic_\an(X_0)/\Pic_\an^0(X_0)=\Pic_\alg(X_0)/\Pic^0_\alg(X_0)\simeq
\rH^2(X_0,\Z)$ is the Neron-Severi group\footnote{Notice that
  $\Pic_\an(X_0)=\Pic_\alg(X_0)$ by the GAGA correspondence since
  $X_0$ is a projective variety.} of $X$.

Holomorphic line bundles $\cL$ on the elliptic curve $X_0$ are
classified by their factors of automorphy, which by the Appell-Humbert
theorem \cite{BL} are determined by pairs $(H,\chi)$, where $H$ is a
Hermitian pairing on $\C$ whose imaginary part takes integer values on
the lattice $\Lambda$ and $\chi$ is a semicharacter of $\Lambda$
relative to $H$. Let $p_0\in X_0$ be the $\mod\, \Lambda$ image of the
point $z_0=\frac{1+\i}{2}\in \C$. Then the holomorphic line bundle
$\cL=\cO_{X_0}(p_0)$ on $X_0$ has positive unit degree $\deg
\cL=\int_{X_0} c_1(\cL)=+1$ and is described by the data
$H(z,z')=z\overline{z'}$ and $\chi(n_1+\i n_2)=(-1)^{n_1n_2}$, where
we used the fact that the modulus of $X_0$ equals the imaginary
unit. Up to multiplication by a non-zero complex number, there exists
a unique holomorphic section of $\cO_{X_0}(p_0)$ which vanishes at
$p_0$. A conveniently normalized section $s_0$ is described by the
Riemann-Jacobi theta function (traditionally denoted $\vartheta_{00}$
or $\vartheta_3$) at modulus $\tau=\i$:
\be
\vartheta(z)=\vartheta_{00}(z)|_{\tau=\i}=\sum_{n\in \Z}e^{-\pi n^2+2\pi \i n z}~~,
\ee
which satisfies: 
\be
\vartheta(z+1)=\vartheta(z)~~,~~\vartheta(z\pm \i)=e^{\pi \mp 2\pi \i z} \vartheta(z)~~
\ee
and vanishes at the points $\frac{1+\i}{2}+\lambda$ with $\lambda\in \Lambda$. 
The class of $\cL$ modulo $\Pic^0_\an(X_0)$ generates the Neron-Severi
group of $X_0$. Hence the pullbacks $L_\pm:=(\bvarphi^\pm)^\ast(\cL)$
are holomorphic line bundles defined on $X$ such that each of
$c_1(L_+)$ and $c_1(L_-)$ generates the group $\rH^2(X,\Z)\simeq
\Z$. Since $\bvarphi^+$ and $\bvarphi^-$ differ by composition with
the map $(x_1,x_2)\rightarrow (x_1,x_2^{-1})$ (which induces an
orientation reversal of $T^2$), we in fact have $c_1(L_-)=-c_1(L_+)$
and hence $L_-\simeq L_+^{-1}$. Notice that the restriction of $L_+$
to $T^2$ has positive degree with respect to the complex structure
induced by the diffeomorphisms $\bvarphi_0^+$ while the restriction of
$L_-$ to $T^2$ has positive degree with respect to the opposite
complex structure, which is induced by the diffeomorphisms
$\bvarphi_0^-$. The $\bvarphi^\pm$-pullbacks of $s_0$ give global
holomorphic sections $s_\pm\in \rH^0(L_\pm)$. These are described by
the $\Z^2$-quasiperiodic holomorphic functions $f_\pm\in \O(\C^2)$
defined through:
\be
f_\pm(z_1,z_2)\eqdef \vartheta(z_1\pm \i z_2)=\sum_{n\in \Z}e^{-\pi n^2+2\pi \i n (z_1\pm \i z_2)}~~.
\ee
The zero divisors $D_\pm$ of $s_\pm$ on $X$ meet the 2-torus
$T^2\subset X$ at the single point given by $x_1=x_2=-1$,
which corresponds to the point $p_0$ of $X_0$ through each of the
diffeomorphisms $\bvarphi_0^\pm$. The tensor product $s_+\otimes
s_-\in \rH^0(L_+\otimes L_-)$ is described by the holomorphic function
$f\in \O(\C^2)$ given by:
\be
f(z_1,z_2)\eqdef f_+(z_1,z_2)f_-(z_1,z_2)~~, 
\ee
which satisfies: 
\be
f(z_1+1,z_2)=f(z_1,z_2)~~,~~f(z_1,z_2+1)=e^{2\pi +4\pi z_2} f(z_1,z_2)~~.
\ee
The isomorphism $L_+\otimes L_-\simeq \cO_X$ is realized on
$\Z^2$-factors of automorphy defined on $\C^2$ by the holomorphic
function $S:\C^2\rightarrow \C^\ast$ given by:
\be
S(z_1,z_2)\eqdef e^{-2\pi z_2^2}~~,
\ee
which satisfies $\frac{S(z_1,z_2)}{S(z_1,z_2+1)}=e^{2\pi +4\pi
  z_2}$. The section $s_+\otimes s_-\in \rH^0(L_+\otimes L_-)$
corresponds through this isomorphism to an element $W\in
\rH^0(\cO_X)=\O(X)$ whose lift to $\C^2$ is the
$\Z^2$-periodic function:
\be
\tilde{W}(z_1,z_2)\eqdef S(z_1,z_2)f(z_1,z_2)=e^{-2\pi z_2^2} \vartheta(z_1+\i z_2)\vartheta(z_1-\i z_2)~~.
\ee
Applying the construction of Subsection \ref{subsec:nontriv} with
$L:=L_+$, $L^{-1}\simeq L_-$ and $u=s_+$, $v=s_-$ gives a topologically
non-trivial elementary holomorphic factorization $(E,D)$ of $W$, where
$E=\cO_X\oplus L_+$ and $D=\left[\begin{array}{cc} 0 & s_-\\ s_+
    & 0\end{array} \right]$.
\end{example}

\

\begin{remark}
In \cite{EP}, Eisenbud and Peeva studied the infinite minimal free
resolution of finitely-generated modules $M$ over local complete
intersection rings $R=S/\langle f_1, \dots, f_d\rangle$, where $S$ is a regular
local ring and $f_1,\ldots, f_d$ is a regular sequence of elements of
$R$, showing that sufficiently high syzygies of $M$ are determined by
``higher matrix factorizations''.  This extends the correspondence of
\cite{Eisenbud} from local hypersurfaces to local complete
intersections. For an affine algebraic LG pair $(X_\alg,W_\alg)$
(namely $X_\alg$ is an affine manifold and $W_\alg$ is a regular
function defined on $X_\alg$), the result of \cite{Eisenbud} combines
\cite{Orlov} with the Buchweitz correspondence \cite{Buchweitz} to
give an equivalence between the category of {\em algebraic} matrix
factorizations of $W_\alg$ over $\C[X_\alg]$ and the category of
singularities of the quotient $\C[X_\alg]/\langle W_\alg\rangle$, where
$\C[X_\alg]$ is the affine coordinate ring of $X_\alg$. This setting
is quite different from our non-Noetherian situation, which concerns
holomorphic matrix factorizations of a single holomorphic function $W$
over a Stein manifold $X$. While this subject lies outside the scope
of the present paper, notice that the results of \cite{EP} could be
applied in our holomorphic setting when $X$ is a Stein manifold which
is locally an analytic complete intersection. However, relating the
category of holomorphic factorizations of a Stein LG pair $(X,W)$ to
an appropriate holomorphic version of the category of singularities of
$W$ seems to require an appropriate generalization of the
Buchweitz-Orlov correspondence (see \cite{Orlov,Buchweitz}) to our
non-Noetherian situation.
\end{remark}

\

\subsection{Complements of anticanonical divisors} 
\label{subsec:complements}

There is a simple way to produce many examples of Stein manifolds to
which our results can be applied.

\

\begin{Proposition}
\label{CY}
Let $Y$ be a non-singular complex projective Fano variety and $D$
be a smooth anticanonical divisor on $Y$. Then the complement
$X:=Y \setminus D$ is a non-compact Stein Calabi-Yau manifold.
\end{Proposition}

\begin{proof}
The vanishing of the first Chern class follows from the adjunction
formula. The Stein property follows since the complement of $D$
is affine and since (the analytification) of any affine
variety is Stein. \qed
\end{proof}

\

\noindent The Stein manifolds produced by Proposition \ref{CY} are
analytifications of non-singular affine varieties. Notice, however,
that the Landau-Ginzburg models considered here on such manifolds
differ from the algebraic Landau-Ginzburg models considered in
\cite{Orlov}. In particular, the superpotential $W$ is a holomorphic
(rather than a regular algebraic) function and our category of
topological D-branes is described by holomorphic (rather than
algebraic) matrix factorizations of $W$.

\

\begin{remark} 
Mirror symmetry for the A-model defined on complements $X$ of
anticanonical divisors was studied in \cite{Auroux}. By contrast, the
examples of this section concern the B-type Landau-Ginzburg model defined on
such complements.
\end{remark}

\

\begin{example}
Let $\P^d$ be the $d$-dimensional projective space with $d\geq 2$.  In
this case, we have $K_{\P^d}=\O_{\P^d}(-d-1)$. Let $Z$ be an
irreducible smooth hypersurface of degree $d+1$ in $\P^d$. Notice that
$Z$ defines an anticanonical divisor in $\P^d$ (which is a Fano
manifold). The complement $X=\P^d \setminus Z$ is a Stein Calabi-Yau
manifold. In this example, we have $\rH_1(X,\Z)=\Z_{d+1}$ (see
\cite[Proposition 4.1, page 61]{Dimca}).  By the universal coefficient
theorem \cite[Section 3.1]{hatcher} we conclude that the torsion part
of $\rH^2(X,\Z)$ is isomorphic with
$\Ext^1(\rH_1(X,\Z),\Z)=\Ext^1(\Z_{d+1},\Z) \simeq \Z_{d+1}$. Thus $X$
admits non-trivial holomorphic line bundles and, by the construction
of Subsection \ref{subsec:nontriv}, it also admits topologically
non-trivial elementary holomorphic factorizations.
\end{example}

\

\subsection{Total space of a holomorphic vector bundle}
\label{subsec:totspace}

Let $Y$ be a Stein manifold and $\pi: E \rightarrow Y$ be a
holomorphic vector bundle. Then the total space $X$ of $E$ is Stein 
(see \cite[Chap. V.1., Exercise 2]{FG}). 
The complex tangent bundle $\mathcal{T} X$ of $X$ is an extension:
\be
0 \rightarrow \pi^* E \rightarrow \mathcal{T} X \rightarrow \pi^*
\mathcal{T}Y \rightarrow 0~~,
\ee
hence the Chern polynomial $c_t(\mathcal{\cT} X)$ equals $\pi^* \left(
c_t(E)) \cdot \pi^\ast(c_t(\mathcal{T}Y) \right)$. The projection
$\pi$ is a homotopy equivalence since each fiber of $E$ is
contractible to a point; hence $\pi^\ast:\rH(Y,\Z)\rightarrow
\rH(X,\Z)$ is an isomorphism of groups. In particular, the first
Chern class of $X$ vanishes iff $c_1(E)+c_1(Y)=0$
(i.e. $c_1(E)=c_1(K_Y)$), in which case $X$ is a Stein Calabi-Yau
manifold.

\

\begin{remark}
\label{rem:Kpullback}
Suppose that $K_Y$ is non-trivial. Then the holomorphic
line bundle $L\eqdef \pi^\ast(K_Y)$ on $X$ is non-trivial with first
Chern class $c_1(L)=\pi^\ast(c_1(K_Y))$.
\end{remark}

\

\noindent A particular case of this
construction is obtained by taking $E$ to be the canonical line bundle
$K_Y$ of $Y$. In this case, the first Chern class of $X$ vanishes
automatically. For example, let $Y\subset \P^n$ be the complement
of an algebraic hypersurface $Z\subset \P^n$. Then the
total space $X$ of $K_Y$ is Stein and Calabi-Yau. 

\

\begin{example}
Consider the (non-anticanonical) hypersurface $Z=\{[x,y,z] \in \P^2 \,
| \, x^2+y^2+z^2=0 \}$ in $\P^2$. Then the complement $Y\eqdef
\P^2\setminus Z$ is Stein but not Calabi-Yau (see \cite{Fo}). Moreover, it
is shown in loc. cit that $\rH^2(Y,\Z)\simeq \Z_2$ and that
$c_1(Y)=\gamma$, where $\gamma$ is the generator of
$\rH^2(Y,\Z)$. Thus $K_Y$ is non-trivial with
$c_1(K_Y)=-\gamma$. Applying the construction above, consider the
total space $X$ of $K_Y$. Then $X$ is Stein and Calabi-Yau.  Moreover,
Remark \ref{rem:Kpullback} shows that the pull-back $L$ of $K_Y$ to
$X$ has non-trivial first Chern class. By the construction of
Subsection \ref{subsec:nontriv}, it follows that the $\Z_2$-graded
holomorphic vector bundle $E=\cO_X\oplus L$ supports holomorphic
factorizations of some non-zero function $W\in \O(X)$, whose
associated projective analytic factorization has non-free underlying
$\O(X)$-supermodule.
\end{example}

\

\section{Conclusions}
\label{sec:conclusions}

Since the mathematical formalism used in the present paper may be
unfamiliar to some readers, it may be useful to recall its physics
origins. It is well-known (see, for example \cite{LL,HV}) that
non-anomalous B-type topological Landau-Ginzburg (LG) models in two
dimensions can be defined for any non-constant holomorphic
superpotential $W:X\rightarrow \C$, where $X$ is a non-compact
K\"ahlerian manifold with trivial holomorphic line bundle \footnote{As
in \cite{HV}, the corresponding `untwisted' models need not be
scale-invariant, so they need not possess a $\mathrm{U(1)}_V$
R-symmetry. However, they do possess a non-anomalous $\mathrm{U(1)}_A$
$\rR$-symmetry.}.

For any such pair $(X,W)$, this was extended in \cite{LG1,LG2} (see
\cite{nlg1} for a mathematically rigorous treatment) to the
open-closed case, by constructing explicitly the boundary coupling of
B-type topological D-branes to the worldsheet bulk action, verifying
BRST invariance and studying the corresponding topological
correlators. As shown in \cite{LG2}, this physics construction allows
one to recover the abstract formalism of open-closed topological field
theories \cite{tft,MS} upon performing a ``partial localizaton'' of
the path integral, thus leading to the expressions which form the basis
of the rigorous treatment given in \cite{nlg1}.

In this paper, we considered the special case when $X$ is a Stein
manifold, showing how the basic objects of the open-closed topological
field theory simplify in this situation. Notice that $X$ need not be
(affine) algebraic.  Indeed, a Stein manifold is biholomorphic with
the vanishing set of a finite collection of holomorphic functions
defined on some complex affine space $\C^N$, but these functions need
not be polynomial. Even when $X$ is affine algebraic, there generally
is no reason (in models without scale invariance) to require $W$
to be a regular function (consider the case $X=\C^d$, with $W$ an entire
function which is not a polynomial function).

Our motivation for considering such models comes from the commonly
held belief (which originated with \cite{WittenMirror} and was later
extended in \cite{HV,Plesser} to the case of models which need not be
scale-invariant) that mirror symmetry can fruitfully be understood as
an equivalence between certain pairs of 2-dimensional topological
field theories. At least when $W$ is holomorphic Morse, it is
generally expected that the mirror of an LG model of the type
considered in \cite{nlg1} and in this paper should be the A-type
non-linear sigma model (NLSM) defined by a Fano manifold
$\breve{X}$. When $\breve{X}$ is toric, this A-type model admits a
linear sigma model description \cite{Witten} and the Abelian duality
argument of \cite{Plesser} applied to the latter shows that the B-type
LG model defined by $(X,W)$ is of the special kind considered in
\cite{HV} (see \cite{CO} for a mathematically rigorous
construction). The situation is much less clear when $\breve{X}$ is
Fano but not toric, since the A-type NLSM defined by $\breve{X}$ does
not immediately admit a toric description\footnote{Though a
description as a complete intersection of hypersurfaces in a toric
variety may sometimes be possible.}. In this case, the mirror B-type
LG model should belong to the general class considered in
\cite{LG1,LG2,nlg1,tserre}. Notice that toric Fano varieties are meager among
all Fano varieties, just like the special models of \cite{HV} are
meager among general B-type LG models.

While the impetus for developing the general theory of B-type LG
models was that of going beyond the linear sigma model framework, it
is natural to ask whether the approach pursued in \cite{nlg1} and in
this paper can shed light on some aspects of the latter, such as
giving a precise effective description to the ``hybrid models'' of
\cite{Witten}. It is likely that a connection to hybrid models does
exist in the more general setting of \cite{nlg1}, which allows the
critical set $Z_W$ to be compact but non-discrete. However,
hybrid models with finite-dimensional on-shell state spaces cannot fit
the more restrictive setting of the present paper for the following
reason. Finite-dimensionality of the on-shell bulk and boundary state
spaces of the general B-type LG models discussed in
\cite{LG1,LG2,nlg1,tserre} requires $Z_W$ to be compact. In the setting of
the present paper, compactness of $Z_W$ and the Stein condition on $X$
imply that $Z_W$ must be finite (see Remark \ref{rem:SteinFinite}). As
a result, the classical vacua of the topological LG model form an
isolated set, which implies that Stein LG models cannot correspond to
the ``hybrid models'' of \cite{Witten}, if one requires
finite-dimensionality of on-shell state spaces.

\

\appendix

\section{Stein manifolds}
\label{app:Stein}

\noindent In this Appendix, we recall the definition and some
properties of Stein manifolds. For more information on Stein
manifolds and Stein spaces we refer the reader to
\cite{GR,OkaGrauert,Oka,BS,Fo,FG,FF,HW}.

\

\noindent Let $X$ be a complex manifold. We say that {\em holomorphic
  functions separate points of $X$} if for every pair of distinct
points $x\neq y$ in $X$, there exists a holomorphic function $f \in
\O(X)$ such that $f(x)\neq f(y)$. The {\em holomorphic (or analytic)
  hull} of a compact subset $K\subset X$ is defined through:
\be
\hat{K}_{\O(X)}\eqdef \{x\in X| \forall f\in \O(X)~:~|f(x)|\leq \max_{y\in K}|f(y)|\}~~,
\ee 
where $\O(X)=\cO_X(X)$ is the $\C$-algebra of complex-valued
holomorphic functions defined on $X$.

\

\begin{Definition} {\rm \cite{FF}} 
A complex manifold $X$ is called {\it holomorphically convex} if
$\hat{K}_{\O(X)}$ is compact for any compact subset $K\subset X$.
\end{Definition}

\vspace{3mm}

\begin{Definition} {\rm \cite{FF}} 
Let $X$ be a complex manifold with $\dim_\C X=d$. We say that $X$ is
{\em Stein} if the following three conditions are satisfied:
\begin{enumerate}[(1)]
\itemsep 0.0em
\item Holomorphic functions separate points of $X$.
\item $X$ is holomorphically convex.
\item For every point $x \in X$ there exist globally-defined
  holomorphic functions $f_1,...,f_d \in \O(X)$ whose differentials
  $\pd f_j$ are linearly independent at $x$ over $\C$.
\end{enumerate} 
\end{Definition}

\vspace{3mm}

\begin{Definition} 
A complex manifold $X$ is called {\em $K$-complete} if for every point
$x\in X$ there exists a holomorphic map $f: X \to \C^N$ for some $N$
such that $x$ is an isolated point of the fiber $f^{-1}(f(x))$.
\end{Definition}

\vspace{3mm}

\begin{Theorem}
A complex manifold $X$ is Stein iff it satisfies the following
properties:
\begin{enumerate}[(1)]
\itemsep 0.0em
\item Holomorphic functions separate points on $X$.
\item $X$ is holomorphically convex.
\item $X$ is $K$-complete.
\end{enumerate}
\end{Theorem}

\vspace{3mm}

\begin{Theorem}
A complex manifold is Stein iff it is biholomorphic to a {\em closed}
complex submanifold of $\C^N$ for some $N$.
\end{Theorem}

\vspace{3mm}

\begin{Theorem} {\rm \cite[page 337]{FF}}
Every Stein manifold of dimension $d>1$ admits a proper holomorphic
embedding into $\C^{N_d}$ for $N_d=\big[\frac{3d}{2}\big]+1$ and a
proper holomorphic immersion into $\C^{M_d}$ for
$M_d=\big[\frac{3d+1}{2}\big]$.
\end{Theorem}

\vspace{3mm}

\begin{remark} 
The $\C$-algebra $\O(X)$ of complex-valued holomorphic functions
defined on a Stein manifold $X$ need not be finitely-generated. It is
known \cite{HW} that $\O(X)$ is finitely-generated iff $X$ can be
embedded as a closed {\em polynomially} convex subset of $\C^N$ for
some $N$.
\end{remark}

\subsection{Cartan-Serre theorems and the Oka-Grauert principle} 

\

\

\begin{Theorem}[Cartan] {\rm \cite[page 124]{GR}}
\label{thm:B}
For every coherent analytic sheaf ${\cal F}$ on a Stein manifold $X$, 
the following statements hold:
\begin{enumerate}[A.]
\item For any $x\in X$, the stalk ${\cal F}_x$ is generated as an
  $\cO_{X,x}$-module by global sections of ${\cal F}$.
\item $\rH^i(X,{\cal F})=0$ for all $i>0$.
\end{enumerate}
\end{Theorem}

\vspace{3mm}

\noindent Since $\rH^{i,j}(X)=\rH^j(X,\wedge^i TX)$, this implies:

\

\begin{Corollary}
On any Stein manifold $X$ the Dolbeault cohomology groups 
$\rH^{i,j}(X)$ vanish for all $i\geq 0$ and $j\geq 1$.
\end{Corollary}

\

\noindent The following result is known as the ``Oka-Grauert principle'':

\

\begin{Theorem} {\rm \cite{OkaGrauert}}
The holomorphic and topological classifications of holomorphic vector
bundles over a Stein manifold coincide.
\end{Theorem}

\vspace{3mm}

\begin{Corollary}
\label{cor:SteinCY}
Let $X$ be a Stein manifold. Then $X$ is Calabi-Yau iff $c_1(TX)=0$. 
\end{Corollary}

\begin{proof}
$X$ is Calabi-Yau iff its canonical line bundle $K_X$ is holomorphically trivial. Since $X$ is
  Stein, the Oka-Grauert principle shows that $K_X$ is holomorphically
  trivial iff it is topologically trivial, i.e. iff $c_1(K_X)=0$.
  Since $c_1(K_X)=-c_1(TX)$, this amounts to the condition
  $c_1(TX)=0$. 
\end{proof}

\subsection{Holomorphically parallelizable Stein manifolds}

A special class of Stein manifolds consists of those complex manifolds
which can be embedded as analytic complete intersections in some complex
affine space. Such Stein manifolds are always Calabi-Yau. In
fact, they coincide with the class of holomorphically parallelizable
Stein manifolds, by the results of \cite{Fo} recalled below.

\

\begin{Definition}
A complex manifold $X$ with $\dim_\C X=d$ is called {\em
  holomorphically parallelizable} if the holomorphic tangent bundle
$TX$ is holomorphically trivial, i.e. it is isomorphic with the trivial
holomorphic vector bundle of rank $d$ in the category of
holomorphic vector bundles over $X$.
\end{Definition}

\

\noindent It is a consequence of the Oka-Grauert principle \cite{OkaGrauert,Oka} that a
Stein manifold $X$ is holomorphically parallelizable iff $TX$ is {\em
  topologically} trivial (see \cite{Fo}). Notice that any
holomorphically parallelizable manifold is Calabi-Yau.
 The following results were established in \cite{Fo}:

\

\noindent A $d$-dimensional analytic submanifold $X$ of $\C^N$ is
called an {\em analytic complete intersection in $\C^N$} if the ideal
$I_N(X)\subset \O(\C^N)$ of all holomorphic functions defined on
$\C^N$ and which vanish identically on $X$ can be generated by $N-d$
independent elements, i.e. if there exist $N-d$ holomorphic functions
$f_1,...,f_{N-d}:\C^N\rightarrow \C$ such that:
\be
X=\{x\in\C^N ~|~f_1(x)=\ldots=f_{N-d}(x)=0\}
\ee
and such that the rank of the matrix $(\frac{\partial f_i(x)}{\partial
  x_j})\in \Mat(N-d,N,\C)$ equals $N-d$ at every point $x\in X$.

\

\begin{Theorem} 
A Stein manifold is holomorphically parallelizable iff it is
biholomorphic with an analytic complete intersection in $\C^N$ for
some $N$.
\end{Theorem}

\vspace{3mm}

\bl Let $X$ be a complex-analytic submanifold\,\footnote{It is {\em not}
  assumed that $X$ is an analytic complete intersection inside this $\C^N$.} of
$\C^N$ with $\dim_\C X=d$. Then the following statements hold:
\begin{description}
\item{(i)} If the normal bundle of $X$ is trivial, then $X$ is
  holomorphically parallelizable.
\item{(ii)} If $X$ is holomorphically parallelizable and $N\geq  \frac{3d}{2}$,
  then the normal bundle of $X$ is trivial.
\end{description}
\el

\vspace{3mm}

\begin{Theorem}
\label{thm:CI}
Let X be an analytic submanifold of $\C^N$ with $\dim_\C X=d$.
\begin{description}
\item{(i)} If $X$ is an analytic complete intersection in $\C^N$, then
  it is holomorphically parallelizable.
\item{(ii)} If $X$ is holomorphically parallelizable and $N\geq
  \frac{3d}{2}\!+\!1$, then $X$ is a complete intersection in $\C^N$.
\end{description}
\end{Theorem}

\vspace{3mm}

\begin{Corollary}
\label{cor:SteinHyp}
Let $X$ be an analytic submanifold of $\C^N$ with $\dim_\C X=N-2$.
Then $X$ is holomorphically parallelizable iff the first Chern class
$c_1(X)$ vanishes.
\end{Corollary}

\vspace{3mm}

\begin{Corollary}
Let $N\leq 7$. Then an analytic submanifold $X$ of $\C^N$ is a complete
intersection in $\C^N$ iff $X$ is holomorphically parallelizable.
\end{Corollary}

\vspace{3mm}

\noindent Since for a $d$-dimensional Stein manifold we have
$\rH^i(X,\Z)=0$ for all $i>d$, one can express holomorphic
parallelizability of low-dimensional Stein manifolds completely in
terms of Chern classes. For example:

\vspace{3mm}

\begin{Proposition}{\rm \cite{Fo}}
\label{prop:LowDim}
Let $X$ be a Stein manifold with $\dim_\C X\leq 5$. Then $X$ is
holomorphically parallelizable iff $c_1(X)=c_2(X)=0$. 
\end{Proposition}

\vspace{3mm}

\begin{Theorem} {\rm \cite[page 126]{GR}}
Every $d$-dimensional Stein manifold can be biholomorphically 
mapped onto a closed complex submanifold of $\C^{2d+1}$.
\end{Theorem}

\

\begin{acknowledgements}
This work was supported by the research grant IBS-R003-S1, 
"Constructive string field theory of open-closed topological B-type strings". 
\end{acknowledgements}


\begin{thebibliography}{100}
\bibitem{LG1}{Lazaroiu, C.I.: {On the boundary coupling of
topological Landau-Ginzburg models}. JHEP {\bf 05}, 037 (2005)}
\bibitem{LG2}{Herbst, M., Lazaroiu, C. I.: {Localization and traces in
open-closed topological Landau-Ginzburg models}. JHEP {\bf 05}, 044
(2005)}
 \bibitem{tft}{Lazaroiu, C.I.: {On the structure of open-closed
topological field theories in two dimensions}. Nucl. Phys. B {\bf
603}, 497--530 (2001) }
\bibitem{MS}{Moore, G., Segal, G.: {D-branes and K-theory in 2D
topological field theory}. arXiv:hep-th/0609042}
\bibitem{LP}{Lauda, A.D., Pfeiffer, H.: {Open-closed strings:
two-dimensional extended TQFTs and Frobenius algebras}. Topology
Appl. {\bf 155}(7), 623--666 (2008)}
\bibitem{nlg1}{Babalic, E.M., Doryn, D., Lazaroiu, C.I., Tavakol,
M.: {Differential models for B-type open-closed topological
Landau-Ginzburg theories}. Commun. Math. Phys. (2018), 1-66.
https://doi.org/10.1007/s00220-018-3137-5. [arXiv:1610.09103v3]}
\bibitem{tserre}{Doryn, D., Lazaroiu, C.I.: {Non-degeneracy of
cohomological traces for general Landau-Ginzburg
models}. arXiv:1802.06261 [math.AG]}
\bibitem{Vafa}{Vafa, C.: {Topological Landau-Ginzburg models}.
Mod. Phys. Lett. A {\bf 6}, 337--346 (1991)}
\bibitem{KL1}{Kapustin, A., Li, Y.: {Topological correlators in
Landau-Ginzburg models with boundaries}.  Adv. Theor. Math. Phys. {\bf
7}(4), 727--749 (2003)}
\bibitem{D}{Dyckerhoff, T.: {Compact generators in categories of
matrix factorizations}. Duke Math. J. {\bf 159}(2), 223--274 (2011)
}
\bibitem{DM}{Dyckerhoff, T., Murfet, D.: {The Kapustin-Li formula
revisited}. Adv. Math. {\bf 231}(3-4), 1858--1885 (2012)}
\bibitem{PV}{Polishchuk, A., Vaintrob, A.: {Chern characters and
    Hirzebruch-Riemann-Roch formula for matrix factorizations}. Duke
  Math. J. {\bf 161}(10), 1863--1926 (2012)}
\bibitem{LLS}{Li, C., Li, S., Saito, K.: {Primitive forms via
    polyvector fields}. arXiv:1311.1659v3 [math.AG] }
\bibitem{S}{Shklyarov, D.: {Calabi-Yau structures on categories
    of matrix factorizations}. J. Geom. Phys. {\bf 119}, 193--207 (2017)}
\bibitem{ForsterStein}{Forster, O.: {Zur Theorie der Steinschen
    Algebren und Moduln}. Math. Z. {\bf 97}, 376--405 (1967) }
\bibitem{Morye}{Morye, A.S.: {Note on the Serre-Swan theorem}.
  Math. Nachrichten {\bf 286}(2-3), 272--278  (2013)}
\bibitem{AN}{Aprodu, M., Nagel, J.: {Koszul cohomology and algebraic
    geometry}. University Lecture Series {\bf 52}, Amer. Math. Soc. (2010)}
\bibitem{GR}{Grauert, H., Remmert, R.: {Theory of Stein spaces}.
  Classics in Mathematics, Springer, Berlin, Heidelberg (2004)}
 \bibitem{BoTu}{Bott, R., Tu,  L.W.: {Differential forms in algebraic
    topology}. In: Graduate Texts in Mathematics, vol. {\bf 82}, Springer, New York (1982)}
 \bibitem{Wei}{Weibel, C.: {Cyclic homology for schemes}. 
Proceedings of AMS {\bf 124}(6), 1655--1662 (1996)}
\bibitem{OkaGrauert}{Grauert, H.: {Analytische Faserungen \"{u}ber
    holomorph-vollst\"{a}ndigen R\"{a}umen}. Math. Ann. {\bf 135},
   263--273 (1958)}
\bibitem{Oka}{Forsteneric, F.,  Larusson, F.: {Survey of Oka theory},
  New York J. Math. {\bf 17a}, 1--28  (2011)}
\bibitem{picard}{Hamm, H.A., L\^e, D.T.: {On the Picard group for
non-complete algebraic varieties}. Singularit\'es
Franco-Japonaises, S\'eminaires et Congr\`es {\bf 10},
71--86, Soci\'et\'e Math\'ematique de France, Paris (2005)}
\bibitem{Lelong}{Lelong, P., Gruman, L.: {Entire functions of several
complex variables}. In: Grundlehren der mathematischen Wissenschaften, vol. {\bf 282}. 
Springer, Berlin, Heidelberg (1986)}
\bibitem{BS}{Behnke, H., Stein, K.: {Entwicklungen analytischer
Funktionen auf Riemannschen Fl\"achen}. Math. Annalen {\bf 120}
(1947)}
\bibitem{FO}{Forster, O.: {Lectures on Riemann surfaces}. In: Graduate Texts in 
Mathematics, vol. {\bf 81}. Springer, New York (1981)}
\bibitem{Helmer2}{Helmer, O.: {The elementary divisor theorem for
certain rings without chain conditions}. Bull.  Amer. Math. Soc. {\bf
49}, 225--236 (1943)}
\bibitem{Henriksen3}{Henriksen, M.: {Some remarks on elementary
divisor rings, II}. Michigan Math. J. {\bf 3}, 159--163 (1955/56)}
\bibitem{Alling1}{Alling, N.: {The valuation theory of meromorphic
function fields over open Riemann surfaces}. Acta Math. {\bf 110},
79--96 (1963)}
\bibitem{Alling2}{Alling, N.: {The valuation theory of meromorphic
function fields}. Proc. Sympos. Pure Math. {\bf 11}, Amer. Math. Soc.,
8--29 (1968)}
\bibitem{FS}{Fuchs, L., Salce, L.: {Modules over non-Noetherian
domains}. Math. Surveys and Monographs {\bf 84},
Amer. Math. Soc.  (2001)}
\bibitem{bezout}{Doryn, D., Lazaroiu, C.I., Tavakol, M.: {Elementary
matrix factorizations over B\'ezout domains}. arXiv:1801.02369 [math.AC]}
\bibitem{BHLS}{Brunner, I., Herbst, M., Lerche, W., Scheuner, B.: 
{Landau-Ginzburg realization of open string TFT}. JHEP {\bf 11}, 043
(2006)}
\bibitem{KL2}{Kapustin, A., Li, Y.: {D-branes in Landau-Ginzburg
models and algebraic geometry}. JHEP {\bf 12}, 005 (2003)}
\bibitem{HLL}{Herbst, M., Lazaroiu, C.I., Lerche, W.: {D-brane
effective action and tachyon condensation in topological minimal
models}. JHEP {\bf 03}, 078 (2005)}
\bibitem{edd}{Doryn, D., Lazaroiu, C.I., Tavakol, M.: {Matrix factorizations 
over elementary divisor domains}. arXiv:1802.07635 [math.AC]}
\bibitem{Dimca}{Dimca, A.: {Hyperplane arrangements -- An introduction}. Universitext.
Springer (2017)}
\bibitem{Yuzvinsky}{Yuzvinsky, S.A.: {Orlik-Solomon algebras in
algebra and topology}. Russ. Math. Surv. {\bf 56}(2) 293--364 (2001)}
\bibitem{DG}{Docquier, F., Grauert, H.: {Levisches Problem und
Rungescher Satz f\"ur Teilgebiete Steinscher Mannigfaltigkeiten}.
Math. Ann. {\bf 140}, 94--123 (1960)}
\bibitem{Brieskorn}{Brieskorn, E.: {Sur les groupes de tresses}.
Lecture Notes in Math. {\bf 317}, 21-44 (1973)}
\bibitem{OS}{Orlik, P., Solomon, L.: {Combinatorics and topology of
complements of hyperplanes}. Invent. Math. {\bf 56}, 167--189 (1980)}
\bibitem{Weibel}{Weibel, C.A.: {Pic is a contracted functor}.
Invent. Math. {\bf 103}, 351--377 (1991)}
\bibitem{Fo}{Forster, O.: {Some remarks on parallelizable Stein
manifolds}. Bull. Amer. Math. Soc. {\bf 73} (1967)}
\bibitem{quiver}{Kajiura, H., Saito, K., Takahashi, A.: {Matrix 
Factorizations and Representations of Quivers II: 
type ADE case}. Adv. Math. {\bf 211}(1), 327--362 (2007)}
\bibitem{BL}{Birkenhake, C., Lange, H.: {Complex Abelian varieties}.
 In: Grundlehren der mathematischen Wissenschaften, vol. {\bf 302}. 
 Springer, Berlin, Heidelberg (2004)}
\bibitem{EP}{Eisenbud, D., Peeva, I.: {Minimal free resolutions over complete 
intersections}. Lecture Notes in Mathematics {\bf 2152} (2016)}
  \bibitem{Eisenbud}{Eisenbud, D.: {Homological algebra on a complete
intersection, with an application to group representations}.
Trans. Amer. Math. Soc. {\bf 260}, 35--64 (1980)}
\bibitem{Orlov}{Orlov, D.: {Triangulated categories of singularities
and D-branes in Landau-Ginzburg models}. Proc. Steklov
Inst. Math. {\bf 246}(3), 227--248 (2004)}
\bibitem{Buchweitz}{Buchweitz, R.O.: {Maximal Cohen-Macaulay modules
and Tate cohomology over Gorenstein rings} (1986). 
https://tspace.library.utoronto.ca/handle/1807/16682}
\bibitem{Auroux}{Auroux, D.: {Mirror symmetry and T-duality in the
complement of an anticanonical divisor}. J. G\"okova Geom.  Topol. GGT
{\bf 1}, 51--91 (2007)}
\bibitem{hatcher}{Hatcher, A.: {Algebraic topology}. Cambridge
University Press, Cambridge (2002)}
\bibitem{FG}{Fritzsche, K, Grauert, H.: {From Holomorphic Functions to
Complex Manifolds}. In: Graduate Texts in Mathematics {\bf 213}. Springer, New York
(2002) }
\bibitem{LL}{Labastida, J.M.F., Llatas, P.M.: {Topological matter
in two-dimensions}. Nucl.  Phys. B {\bf 379}, 220 (1992)}
\bibitem{HV}{Hori, K., Vafa, C.: {Mirror symmetry}. 
arXiv:hep-th/0002222}
\bibitem{WittenMirror}{Witten, E.: {Mirror manifolds and topological
field theory}. In: Yau, S.T. (ed.) Essays on mirror manifolds, pp.120--158, 
Internat. Press, Hong Kong (1992)}
\bibitem{Plesser}{Morrison, D., Plesser, R.: {Towards mirror symmetry
as duality for two-dimensional abelian gauge theories}.
Nucl. Phys. Proc. Suppl. {\bf 46}, 177--186 (1996) }
\bibitem{Witten}{Witten, E.: {Phases of $N=2$ theories in
two-dimensions}. Nucl. Phys. B {\bf 403}, 159--222 (1993)}
\bibitem{CO}{Cho, C.-H., Oh, Y.-G.: {Floer cohomology and disc
instantons of Lagrangian torus fibers in Fano toric manifolds}. Asian
J. Math. {\bf 10}(4), 773--814 (2006)}
\bibitem{FF}{Forsteneric, F.: {Stein manifolds and holomorphic
mappings -- The homotopy principle in complex analysis}. In: Ergebnisse
der Mathematik und ihrer Grenzgebiete, vol. {\bf 56}. Springer, Berlin, Heidelberg (2011)}
\bibitem{HW}{Heal, E.R., Windham, M.P.: {Finitely-generated
    F-algebras with applications to Stein manifolds}. Pacific J. Math.
  {\bf 51}(2), 459--465 (1974)}

\end{thebibliography}
\end{document}